\documentclass[reqno]{amsart}
\usepackage{amsmath}
\usepackage{mathrsfs,amscd}
\usepackage{paralist}
\usepackage[colorlinks=true]{hyperref}
\hypersetup{urlcolor=blue, citecolor=red}
\usepackage[numbers]{natbib}

\newtheorem{theorem}{Theorem}[section]

\newtheorem{lemma}[theorem]{Lemma}
\newtheorem{proposition}{Proposition}

\theoremstyle{definition}

\newtheorem{remark}{Remark}

\let\bs\boldsymbol
\def\C{\mathscr{C}}

\def\K{\mathbb{K}}
\def\RR{\mathbb{R}}

\newcommand{\Div}{\nabla\!\cdot\!}
\newcommand{\Rot}{\nabla\!\times\!}
\newcommand{\funcao}[3]{\mbox{$#1 : #2 \longrightarrow #3$}}
\newcommand{\lraup}{\relbar\joinrel\relbar\joinrel\rightharpoonup}
\newcommand{\lra}{\relbar\joinrel\relbar\joinrel\rightarrow}

\newcommand{\fim}{\hfill{$\square$}}
\newcommand{\Tende}[3]{\ensuremath{\raisebox{-3pt}{$\begin{CD}{#1}@>>{#2}>{#3}\end{CD}$}}}
\newcommand{\norm}[1]{\left\Vert#1\right\Vert}
\newcommand{\W}[2]{{\mathbb W}^{#1}(#2)}
\newcommand{\Wtilde}[2]{{\widetilde{\mathbb W}}^{#1}(#2)}

\title[On a p-curl system arising in electromagnetism]%
{On a p-curl system arising in electromagnetism}

\author[Fernando Miranda, Jos\'e-Francisco Rodrigues and Lisa Santos]{}


\thanks{The first an last authors are supported by the Research Centre of Mathematics of the
University of Minho through the FCT Pluriannual Funding Program and
FCT project \mbox{UT-Austin/MAT/0035/2008}}

\begin{document}
\maketitle

\centerline{\scshape Fernando Miranda }
\medskip
{\footnotesize
 \centerline{Department of Mathematics/CMAT, University of Minho}
 \centerline{Campus de Gualtar, 4710-057 Braga, Portugal}
 \centerline{fmiranda@math.uminho.pt}
}

\medskip

\centerline{\scshape Jos\'e-Francisco Rodrigues}
\medskip
{\footnotesize
 \centerline{CMAF/FCUL, University of Lisbon}
 \centerline{Av.\ Prof.\ Gama Pinto, 2, 1649-003 Lisboa, Portugal}
 \centerline{rodrigue@fc.ul.pt}
}

\medskip

\centerline{\scshape Lisa Santos}
\medskip
{\footnotesize
 \centerline{Department of Mathematics/CMAT, University of Minho}
 \centerline{Campus de Gualtar, 4710-057 Braga, Portugal}
 \centerline{lisa@math.uminho.pt}
}

\bigskip

\begin{abstract}
We prove existence of solution of a $p$-curl type evolutionary system arising in
electromagnetism with a power nonlinearity of order $p$, $1<p<\infty$, assuming
natural tangential boundary conditions.
We consider also the asymptotic  behaviour in the
power obtaining, when $p$ tends to infinity, a variational inequality with a
curl constraint.
We also discuss the existence, uniqueness and continuous dependence on the data of the solutions to
general variational inequalities with curl constraints dependent on time, as
well as the asymptotic stabilization in time towards the stationary solution
with and without constraint.
\end{abstract}

\section{Introduction}

We consider a nonlinear electromagnetic field  in a bounded domain $\Omega$
of $\RR^3$.
The electric and the magnetic fields, respectively $\bs e=\bs e(x,t)$ and $\bs h=\bs h(x,t)$, and the electric and magnetic inductions, respectively
$\bs d(x,t)$ and $\bs b=\bs b(x,t)$,
satisfy the  Maxwell's equations ($\partial_t=\frac{\partial\ }{\partial t}$, $\Rot{}=\mbox{curl}$, $\Div{}=\mbox{div}$)
\begin{eqnarray}\label{MaxwellEquations1}
\nonumber \partial_t \bs d+\bs j &= &\Rot\bs h,\\
 \partial_t\bs b+\Rot\bs e&=&\bs f,\\
\nonumber \Div\bs d&=&q,\\
\nonumber\Div\bs b&=&0
\end{eqnarray}
where $\bs j$ denotes the total current density, $q$ is the electric charge and $\bs f$, which is zero in the classical setting, is here a given internal magnetic current (see \cite{Bossavit1998,LandauLifshitz1960}).
Denoting  by $\mu$ the magnetic permeability constant, we assume the following constitutive law
\begin{equation*}
 \bs b=\mu\bs h
\end{equation*}
and the following nonlinear extension of Ohm's law,
\begin{equation*}
|\bs j|^{p-2}\bs j=\sigma\bs e,
\end{equation*}
where $\sigma$ is the electric conductivity.

If in the first equation of (\ref{MaxwellEquations1}) we neglect the term $\partial_t\bs d$, the magnetic field $\bs h$ is then divergence free and
$$\mu\,\partial_t\bs h+\Rot\big(\tfrac1\sigma|\Rot\bs h|^{p-2}\Rot\bs h\big)=\bs f.$$

Denoting  $\Gamma=\partial\Omega$ and $\Sigma_T=\Gamma\times(0,T)$, we impose the following natural tangential boundary conditions
\begin{equation*}
 \bs h\cdot \bs n=0\quad\text{and}\quad \bs e\times\bs n=\bs g\quad
 \text{on}\ \Sigma_T,
\end{equation*}
where $\bs n$ denotes the external unitary normal vector to the boundary $\Gamma$.
The boundary condition $\bs h\cdot \bs n=0$ is naturally associated with $\nabla\cdot\bs h=0$ in $Q_T=\Omega\times(0,T)$ and $\bs e\times\bs n=\bs g$ corresponds
to consider a superconductive wall, i.e., a tangent current field.

Recalling the relation between $\bs e$ and $\bs h$, if we set $\nu=\frac1\sigma>0$, we are lead to the  problem
\begin{subequations}\label{pforte}
 \begin{alignat}{3}
 \Div\bs h=0&\quad \text{and }&\quad\mu\,\partial_t\bs h+\Rot(\nu| \Rot\bs h|^{p-2}\Rot\bs h)=\bs f&\qquad \text{in }Q_T,\\
 \bs h\cdot \bs n=0&\quad\text{and}&\quad \nu|\Rot\bs h|^{p-2}(\Rot\bs h )\times\bs n=\bs g&\qquad \text{on }\Sigma_T,\\
 && \bs h(0)=\bs h_0&\qquad \text{in }\Omega.
  \end{alignat}
\end{subequations}

As a necessary condition for the existence of solution of this problem, the external field $\bs f$ must
satisfy $\Div \bs f=0$. Besides, the given field $\bs g$ on $\Sigma_T$ must be tangential
and compatible with $\bs f$, more precisely,
$
 \nabla_\Gamma\cdot\bs g=\bs f\cdot\bs n\quad\text{on }\Gamma,
$
where $\nabla_\Gamma\,\cdot$ denotes the surface divergent (see \cite{MitreaMitreaPipher1997,Mitrea2002,MirandaRodriguesSantos2009}).

We may also consider another constitutive law that arises in type-II superconductors and is known as an
extension of the Bean critical-state model presented in \cite{Prigozhin1996}. In this case the current density cannot
exceed the critical value $\Psi>0$ and we have
\begin{equation*}
\bs e=\begin{cases}
       \nu|\Rot\bs h|^{p-2}\Rot\bs h&\mbox{ if }|\Rot\bs h|<\Psi(x,t),\\ \\
       \big(\nu\,\Psi^{p-2}+\lambda\big)\Rot\bs h&\mbox{ if }|\Rot\bs h|=\Psi(x,t),
      \end{cases}
\end{equation*}
where the parameters $\nu=\nu(x)\geq0$ is a  given function
and $\lambda=\lambda(x,t)\geq0$ can be regarded as a (unknown) Lagrange multiplier.

Some easy calculations (see \cite{Prigozhin1996,MirandaRodriguesSantos2009} for details) leads to the variational inequality, for a.e. $t\in\,(0,T)$,
\begin{multline}\label{iv1}
\int_\Omega \partial_t\bs h(t)\cdot(\bs v-\bs h(t))+\int_\Omega\nu|\Rot\bs h(t)|^{p-2}\Rot\bs h(t)\cdot \Rot(\bs v-\bs h(t))\\
\geq
 \int_\Omega\bs f(t)\cdot (\bs v-\bs h(t))+\int_\Gamma \bs g(t)\cdot (\bs v-\bs h(t)),
\end{multline}
for any test function $\bs v=\bs v(x)$ such that $|\Rot\bs v(x)|\leq \Psi(x,t)$.
This leads to search the solution in the time dependent convex set
\begin{equation}\label{K1}
\K(t)=\{\bs v =\bs v(x): |\Rot\bs v(x)|\leq\Psi(x,t),\ x\in\Omega\}\quad\text{for a.e.}\ t\in(0,T).
\end{equation}

In Section~\ref{sec2} we study the evolutionary problem (\ref{pforte}), showing the existence of a unique solution in the variational framework of quasilinear monotone operators in the appropriate functional subspace of $ W^{1,p}(\Omega)^3$.
We notice that in the case of normal boundary condition ($\bs h\times\bs n=\bs 0$ on $\Sigma_T$) existence results for similar
nonlinear Maxwell's system have been obtained in \cite{Yin2004-2,YinLiBenZou2002}.
But these results with tangential boundary condition ($\bs h\cdot\bs n= 0$ on $\Sigma_T$) are presented here for the first time.
We also prove the asymptotic convergence, as $t\rightarrow\infty$ to the stationary solution of the problem already considered in \cite{MirandaRodriguesSantos2009}.

In Section~\ref{sec3} we derive the Bean-type superconductivity variational inequality model with critical value $\Psi=1$
as the limit case $p\rightarrow\infty$, extending a previous scalar case by \cite{BermudezMunoz-SolaPena2005} and a vectorial case
with normal boundary condition due to \cite{YinLiBenZou2002}.

Finally, in Section~\ref{sec4}, we solve the evolutionary variational inequality (\ref{iv1}) with the time dependent convex set (\ref{K1}),
showing the existence, uniqueness and continuous dependence on the data $\bs f$, $\bs g$, $\bs h_0$ and $\Psi$ of the solution, in the
appropriate setting.
We also discuss the asymptotic convergence of the solution in $\bs L^2(\Omega)$, as $t\rightarrow\infty$, towards the corresponding stationary
solution obtained in \cite{MirandaRodriguesSantos2009}, for $p\geq\frac65$.

\section{The variational equation}\label{sec2}

In what follows $\Omega$ is a bounded, simply connected domain of $\RR^3$ with a $\C^{1,1}$ boundary $\Gamma$. If $E$ denotes a vectorial space, we
denote by $\bs E$ the space $E^3$.

\subsection{The functional framework}

We introduce the functional space
 \begin{equation*}
\W{p}{\Omega}=\big\{\bs v\in\bs W^{1,p}(\Omega):\Div\bs v=0,\ \bs v\cdot\bs n_{|_\Gamma}=0\big\},
 \end{equation*}
$1\le p\le\infty$, which is a closed subspace of the Sobolev space $\bs W^{1,p}(\Omega)$.

\begin{proposition}
For $1< p<\infty$, $\W{p}{\Omega}$ is a reflexive Banach space where the semi-norm $\norm{\,\Rot\bs\cdot\,}_{ \bs L^p(\Omega)}$
is a norm, equivalent to the $\bs W^{1,p}(\Omega)$-norm.
\end{proposition}
\proof For $p>\frac65$ the proof can be found in Theorem 2.1 and Lemma 2.1 of \cite{MirandaRodriguesSantos2009} and for general $1<p<\infty$ in Theorem~2.2 and Corollary~3.3 of  \cite{AmroucheSeloula2010}.\fim

\begin{remark} \label{PoincareTrace}
Two immediate consequences follow from this proposition: there exist  positive constants $C_q$ and $C_r$ such that, given $\bs v\in\W{p}{\Omega}$, the Sobolev inequality
\begin{equation}\label{poincare}
\norm{\bs v}_{\bs L^{q}(\Omega)}\le C_q\norm{\Rot\bs v}_{\bs L^p(\Omega)},
\end{equation}
holds with $q\leq\frac{3p}{3-p}$ if $1<p<3$, any $q<\infty$ if $p=3$ and $q=\infty$ for $p>3$
and the trace theorem
\begin{equation}\label{trace}
 \norm{\bs v_{|_\Gamma}}_{\bs L^r(\Gamma)}\leq C_r\norm{\Rot\bs v}_{\bs L^p(\Omega)},
 \end{equation}
 holds with $r\leq\frac{2p}{3-p}$ if $1<p<3$, any $r<\infty$ if $p=3$ and $r=\infty$ for $p>3$.

In particular, $\bs v\in \bs L^2(\Omega)$ if $p\geq\frac65$.

In what follows the exponents $p,\,q$ and $r$ are related by these Sobolev and trace inequalities.
\hfill{$\square$}\end{remark}

We denote
\begin{equation*}
 \Wtilde{p}{\Omega}=\W{p}{\Omega}\cap\bs L^2(\Omega)
\end{equation*}
and
\begin{equation*}
\bs L^2_\sigma(\Omega) =\overline{\Wtilde{p}{\overset{}{\Omega}}}^{\ \bs L^2(\Omega)}
=\biggl\{\bs v\in \bs L^2(\Omega):\forall\eta\in\C^1(\Omega)\
\int_\Omega \bs v \cdot\nabla\eta=0\biggr\},
\end{equation*}
and we observe that, if $p\geq\frac65$,
$\Wtilde{p}{\Omega}'=\W{p}{\Omega}'$.

\subsection{Existence of solution in the evolution problem}

Let \funcao{\bs a}{Q_T\times\RR^3}{\RR^3} be a Carath\'{e}odory function satisfying
the structural conditions
\begin{subequations}\label{Operador:a:prop}
\begin{align}
\label{Operador:a:prop:a}
\bs a(x,t,\bs u)\cdot \bs u &\geq a_*|\bs u|^p,\\
    \label{Operador:a:prop:b}
|\bs a(x,t,\bs u)|&\leq a^{*}|\bs u|^{p-1}, \\
    \label{Operador:a:prop:c}
\big(\bs a(x,t,\bs u)-\bs a(x,t,\bs v)\big)\cdot(\bs u-\bs v\big)&>0,\
\text{if}\ \bs u\neq\bs v,\\
\tag{\ref*{Operador:a:prop:c}'}
\big(\bs a(x,t,\bs u)-\bs a(x,t,\bs v)\big)\cdot(\bs u-\bs v\big)&\ge
\begin{cases}
a_*|\bs u-\bs v|^p&\text{ if }p\ge 2,\\
a_*\big(|\bs u|+|\bs v|\big)^{p-2}|\bs u-\bs v|^2&\text{ if }p<2,
\end{cases}
\end{align}
\end{subequations}
for given constants $a_*,\, a^*>0$, for all $\bs u,\,\bs v\in\RR^3$ and a.e.\  $(x,t)\in Q_T$.

We consider the following problem:
\begin{subequations}\label{Maxwell:evolutivo:forte}
 \begin{alignat}{3}
 \label{Maxwell:evolutivo:forte:a}
\Div \bs h =0&\quad\text{and}&\quad\partial_t\bs h +\Rot\left(\bs a(x,t,\Rot\bs h)\right) & =\bs f &\quad\text{in}&\quad Q_T,\\
 \label{Maxwell:evolutivo:forte:b}
\bs h\cdot \bs n =0&\quad\text{and}&\quad
\bs a(x,t,\Rot\bs h)\times \bs n & =\bs g &\text{on}&\quad\Sigma_T,\\ 
 \label{Maxwell:evolutivo:forte:d}
&&\bs h(0) & =\bs h_0 &\text{in}&\quad\Omega.
 \end{alignat}
\end{subequations}

Taking (\ref{poincare}) and (\ref{trace}) into account we assume that
\begin{equation}\label{data_f-g}
\bs f\in\bs L^{q'}(Q_T)\quad\text{and}\quad
\bs g\in\bs L^{r'}(\Sigma_T),
\end{equation}
where $q'$ and $r'$ denote the conjugate exponents of $q$ and $r$ respectively, and
\begin{equation}\label{data_h_0}
\bs h_0\in \bs L^2_\sigma(\Omega).
\end{equation}

Hence the following formula of integration by parts
\begin{equation}\label{Green:Rot}
\int_\Omega\Rot\bs a\cdot\bs\varphi - \int_\Omega\bs a\cdot\Rot\bs\varphi =
 \int_\Gamma\bs a\times\bs n\cdot\bs\varphi\quad\forall\bs\varphi\in\bs W^{1,p}(\Omega)
\end{equation}
holds with $\bs a\in\bs L^{p'}(\Omega)$, $\Rot\bs a\in\bs L^{q'}(\Omega)$ and,
in the sense of traces, $\bs a\times\bs n_{|_\Gamma}\in\bs L^{r'}(\Gamma)$  (see \cite{DautrayLions1990} and
\cite{MitreaMitreaPipher1997}).

Whenever $\partial_t\bs h(t)\in \Wtilde{p}{\Omega}'$, interpreting the
integral $\displaystyle\int_\Omega \partial_t\bs h\cdot\bs\varphi$ in the duality sense,
the above formula yields the following weak formulation of the problem (\ref{Maxwell:evolutivo:forte}):
to find $\bs h\in\,L^p(0,T;\Wtilde{p}{\Omega})$ such that, for a.e. $t\in\,(0,T)$,
\begin{gather}\label{Maxwell:evolutivo:variacional}
\begin{split}
\int_\Omega\partial_t\bs h(t)\cdot\bs\varphi + \int_\Omega \bs a(x,t,\Rot\bs h(t))&\cdot\Rot\bs\varphi \\
&=\int_\Omega\bs f(t)\cdot\bs\varphi
+ \int_\Gamma\bs g(t)\cdot\bs\varphi\quad\forall\bs\varphi\in\widetilde{\mathbb W}^p(\Omega)
\\
\bs h(0)&=\bs h_0.
\end{split}
\end{gather}

\begin{proposition}\label{Maxwell:evolutivo}
Suppose that the operator $\bs a$ satisfies the assumptions {\em (\ref{Operador:a:prop}a-c)} and the data and the initial condition satisfy {\em(\ref{data_f-g})} and {\em(\ref{data_h_0})}. Then the problem
{\em (\ref{Maxwell:evolutivo:forte})} has a unique solution $\bs h\in L^p(0,T;\Wtilde{p}{\Omega})\cap C(0,T;\bs L^2_\sigma(\Omega))$
and $\partial_t \bs h \in L^{p'}(0,T;\Wtilde{p}{\Omega}')$.

In addition, there exists a positive constant $C$ such that
\begin{multline}\label{estimate}
\|\bs h\|^2_{L^\infty(0,T;\bs L^2(\Omega))}+\|\Rot\bs h\|_{\bs L^p(Q_T)}^p\\
\le C\left(\|\bs f\|_{\bs L^{q'}(Q_T)}^{p'}+\|\bs g\|_{\bs L^{r'}(\Sigma_T)}^{p'}
+\|\bs h_0\|_{\bs L^2(\Omega)}^2\right).
\end{multline}
\end{proposition}

\begin{proof}
The operator \funcao{A(t)}{\W{p}{\Omega}}{\W{p}{\Omega}'} defined for a.e. $t\in(0,T)$ by
\begin{equation}\label{operador:A:def}
 \langle A(t)\bs h,\bs \varphi\rangle=
\int_\Omega \bs a(x,t,\Rot\bs h)\cdot\Rot\bs \varphi \quad
 \forall\,\bs h,\,\bs \varphi\in\W{p}{\Omega},
\end{equation}
 is a uniformly bounded (independently of $t$), hemicontinuous, monotone and coercive operator, due to the structural properties (\ref{Operador:a:prop}a-c).
Defining, for a.e. $t\in(0,T)$, $L(t)\in\Wtilde{p}{\Omega}'$ by
\begin{equation*}
 \langle L(t),\bs\varphi\rangle=\int_\Omega\bs f(t)\cdot\bs\varphi+
  \int_\Gamma\bs g(t)\cdot\bs\varphi\quad\forall\bs\varphi\in\Wtilde{p}{\Omega},
\end{equation*}
and adapting a well-known existence theorem to monotone operators independent of $t$ (see \cite{Lions1969}), we easily
prove that problem (\ref{Maxwell:evolutivo:variacional}) has a solution in
$L^p(0,T;\Wtilde{p}{\Omega})\cap C(0,T;\bs L^2_\sigma(\Omega))$.

The uniqueness of solution results directly from the strict monotonicity
 (\ref{Operador:a:prop:c}) of the operator $A$.

 To obtain the estimate (\ref{estimate}) choose $\bs h(t)$ as test function in (\ref{Maxwell:evolutivo:variacional}).
 Denoting $Q_t=\Omega\times(0,t)$ and $\Sigma_t=\Gamma\times(0,t)$, we have
\begin{equation*}
\tfrac12\int_\Omega|\bs h(t)|^2+a_*\int_{Q_t}|\Rot\bs h|^p\leq\int_{Q_t}\bs f\cdot\bs h+\int_{\Sigma_t}\bs g\cdot\bs h
 +\tfrac12\int_\Omega |\bs h_0|^2.
 \end{equation*}

Applying H\"{o}lder and Young inequalities and the Remark~\ref{PoincareTrace}, we obtain
\begin{align*}
 \tfrac12\norm{\bs h(t)}^2_{\bs L^2(\Omega)}+
&
a_*\norm{\Rot\bs h}^p_{\bs L^p(Q_t)}\\
&\leq \norm{\bs f}_{\bs L^{q'}(Q_T)} \norm{\bs h}_{\bs L^{q}(Q_T)}
+
\norm{\bs g}_{\bs L^{r'}(\Sigma_T)} \norm{\bs h}_{\bs L^{r}(\Sigma_T)}+
\tfrac12\norm{\bs h_0}^2_{\bs L^2(\Omega)}\\
&\leq C_1 \big(\norm{\bs f}^{p'}_{\bs L^{q'}(Q_T)} + \norm{\bs g}^{p'}_{\bs L^{r'}(\Sigma_T)}\big)\norm{\Rot\bs h}^p_{\bs L^{p}(Q_T)}
+
\tfrac12\norm{\bs h_0}^2_{\bs L^2(\Omega)}
\end{align*}
and the conclusion follows.
\end{proof}

\begin{remark}
The functional framework we introduced provides a general variational setting for the stationary solutions of (\ref{Maxwell:evolutivo:forte}).
Indeed, for instance for arbitrary $\bs f\in\bs L^{q'}(\Omega)$, $\bs g\in\bs L^{r'}(\Gamma)$ and $\nu\in L^{\infty}(\Omega)$, $\nu(x)\ge a_*>0$
for a.e.\ $x\in\Omega$, the unique minimum of the functional in $\W{p}{\Omega}$,
\begin{equation*}
J(\bs h)=\int_\Omega\frac{\nu}{p}\left|\Rot\bs h\right|^p-\int_\Omega\bs f\cdot\bs h-\int_\Gamma\bs g\cdot\bs h,
\end{equation*}
provides the weak stationary solution to (\ref{Maxwell:evolutivo:forte}).
However, as remarked in \cite{MirandaRodriguesSantos2009} in the stationary problem, for the existence of solution of the strong boundary value problem
(\ref{Maxwell:evolutivo:forte}) with given data
$(\bs f,\bs g)$, it is necessary that $\bs f$ is divergence
free and $\bs g$ is tangential and compatible with $\bs f$ ($\nabla_\Gamma\cdot\bs g = \bs f\cdot\bs n$) on $\Gamma$.
But the weak formulation (\ref{Maxwell:evolutivo:variacional}) of the problem (\ref{Maxwell:evolutivo:forte}) has a
unique solution with no restrictions on the data.
\hfill{$\square$}\end{remark}

 Usually, a weak equation is also  a  strong one, as long as it has enough regularity.
 The situation here requires also additional compatibility conditions, since we are working with strongly coupled
 systems and the test functions have strong restrictions (they are divergence free and tangential on the boundary).
 Indeed, given $\bs f\in\bs L^{q'}(\Omega)$, the Helmholtz decomposition
(see \cite{SimaderSohr1992}) gives us that $\bs f=\bs f_0+\nabla\xi$, where $\bs f_0$ is divergence free.
On the other hand, if  $\bs g\in\bs L^{r'}(\Gamma)$, $\bs g=\bs g_T+\bs g_N$, where $\bs g_T$ and $ \bs g_N$ are,
respectively, the tangential and the normal components of $\bs g$.
So, the set of test functions $\Wtilde{p}{\Omega}$  only takes into account $\bs f_0$
(the divergence free component of $\bs f$)  and $\bs g_T$ (the tangential component of
$\bs g$) and consequently the problems (\ref{Maxwell:evolutivo:variacional}) with data $(\bs f,\bs g)$ and
$(\bs f_0,\bs g_T)$ yields the same solution and both correspond to the weak formulation of the problem
(\ref{Maxwell:evolutivo:forte}) with data $(\bs f_0,\bs g_T)$.

In the particular case where
\begin{equation}\label{a-potential}
\bs a(x,t,\bs u)=\nu(x)|\bs u|^{p-2}\bs u,\ \text{with}\ 0<a_*\leq\nu\leq a^*,\ x\in\Omega,
\end{equation}
we can improve the Proposition~\ref{Maxwell:evolutivo} assuming more regularity on the data.

In what follows we denote $\alpha\vee\beta=\max\{\alpha,\beta\}$ and $\alpha\wedge\beta=\min\{\alpha,\beta\}$.

\begin{proposition}\label{te_proposition}
Let $\bs f\in \bs L^{q'\vee2}(Q_T)$,
$\bs g\in L^\infty\big(0,T;\bs L^{r'}(\Gamma)\big)\cap W^{1,r'}\big(0,T;\bs L^{r'}(\Gamma)\big)$ and
 $\bs h_0\in \Wtilde{p}{\Omega}$.
Then the solution $\bs h$ of the problem {\em(\ref{Maxwell:evolutivo:forte})} for
$\bs a$ defined in {\em(\ref{a-potential})} verifies
\begin{equation}\label{te}
 \partial_t\bs h\in \bs L^2(Q_T)\quad\text{and}\quad\Rot\bs h\in L^\infty(0,T;\bs L^{p}(\Omega)).
\end{equation}
\end{proposition}
\begin{proof}
Using Galerkin approximations (see for instance \cite{Lions1969} or Chapter~3 of \cite{Zheng2004}), we may  set formally  $\partial_t\bs h(t)$ as test function in (\ref{Maxwell:evolutivo:variacional}). Integrating between $0$ and $t$ leads to
\begin{multline*}\int_0^t\int_\Omega|\partial_t\bs h|^2+\int_0^t\int_\Omega\nu|\Rot\bs h|^{p-2}\Rot\bs h\cdot\partial_t \Rot\bs h\\
=
 \int_0^t\int_\Omega\bs f\cdot\partial_t\bs h+
 \int_0^t\int_\Gamma\bs g\cdot\partial_t\bs h.
 \end{multline*}

 But
\begin{equation*}
\int_0^t\int_\Gamma\bs g(t)\cdot\partial_t\bs h(t)=\int_\Gamma\bs g(t)\cdot\bs h(t)-\int_\Gamma\bs g(0)\cdot\bs h_0-\int_0^t\int_\Gamma \partial_t\bs g\cdot\bs h(t)
\end{equation*}
so we conclude that
\begin{multline*}
 \left|\int_0^t\int_\Gamma\bs g(t)\cdot\partial_t\bs h(t)\right|\\
 \le C_1\|\bs g\|_{L^\infty(0,T;\bs L^{r'}(\Gamma))}\|\Rot\bs h(t)\|_{L^p(\Omega)}
+\|\partial_t\bs g\|_{L^{r'}(\Sigma_T)}\|\Rot\bs h\|_{\bs L^p(Q_T)}+C_2.
\end{multline*}

Noting that
\begin{equation*}\int_0^t\int_\Omega \nu|\Rot\bs h|^{p-2}\Rot\bs h(t)\cdot\partial_t \Rot\bs h(t)=\tfrac1{p}\int_\Omega\nu|\Rot\bs h(t)|^{p}
  -\tfrac1{p}\int_\Omega\nu|\Rot\bs h_0|^{p},
\end{equation*}
we have
\begin{multline*}
\int_0^t\int_\Omega |\partial_t\bs h|^2
+\tfrac{a_*}{p}\int_\Omega |\Rot\bs h(t)|^p\\
\leq \|\bs f\|_{\bs L^2(Q_T)} \|\partial_t\bs h\|_{\bs L^2(Q_T)}
+C_1\|\bs g\|_{L^\infty(0,T;\bs L^{r'}(\Gamma))}\|\Rot\bs h(t)\|_{\bs L^p(\Omega)}\\
+\|\partial_t\bs g\|_{L^{r'}(\Sigma_T)}\|\Rot\bs h\|_{\bs L^p(Q_T)}+C_2
+\tfrac{a^*}{p}\|\Rot\bs h_0\|_{\bs L^p(\Omega)}^p
\end{multline*}
and so
\begin{multline}\label{te_proposition_estimate}
\norm{\partial_t\bs h}^2_{\bs L^2(Q_T)} +
\norm{\Rot\bs h}^p_{L^\infty(0,T;\bs L^p(\Omega))}\\
\leq
C\big(\|\bs f\|^2_{\bs L^2(Q_T)} + \|\bs g\|^{p'}_{L^\infty(0,T;\bs L^{r'}(\Gamma))}
+\|\partial_t\bs g\|^{p'}_{\bs L^{r'}(\Sigma_T)}\big)+C_3.
\end{multline}
\end{proof}

\subsection{The asymptotic behaviour when $t\rightarrow\infty$}\label{ca-equacao}

In this section we give sufficient conditions in order to establish that
\begin{equation*}
 \norm{\bs h(t)-\bs h_\infty}_{\bs L^2(\Omega)}\longrightarrow 0\quad
\text{as}\quad t\rightarrow\infty,
\end{equation*}
where $\bs h$ denotes the solution of the problem (\ref{Maxwell:evolutivo:forte})
and $\bs h_\infty$ solves the  stationary problem

\begin{alignat*}{3}
\Div \bs h_\infty  =0&\quad\text{and}&\quad\Rot\left(\bs a_\infty(x,\Rot\bs h_\infty)\right) & =\bs f_\infty &\quad\text{in}&\quad \Omega,\\
\bs h_\infty\cdot \bs n =0&\quad\text{and}&\quad
\bs a_\infty(x,\Rot\bs h_\infty)\times \bs n & =\bs g_\infty &\text{on}&\quad\Gamma,
\end{alignat*}
where
\begin{equation*}
\bs f_\infty\in\bs L^{q'}(\Omega)
\quad\text{and}\quad
\bs g_\infty\in\bs L^{r'}(\Gamma),
\end{equation*}
obtaining the variational formulation
\begin{equation}\label{Maxwell:estacionario:variacional}
\int_\Omega \bs a_\infty(x,\Rot\bs h_\infty)\cdot\Rot\bs\varphi=\int_\Omega\bs f_\infty\cdot\bs\varphi+\int_\Gamma \bs g_\infty\cdot\bs\varphi\qquad
\forall\bs\varphi\in\Wtilde{p}{\Omega}
\end{equation}
by applying the integration by parts (\ref{Green:Rot}).

Let
\begin{equation*}
\bs f\in L^\infty(0,\infty;\bs L^{q'}(\Omega))
\quad\text{and}\quad
\bs g\in L^\infty(0,\infty;\bs L^{r'}(\Gamma)),
\end{equation*}
and denote
\begin{equation}\label{xi}
\xi(t)=\norm{\bs f(t)-\bs f_\infty}_{L^{s}(\Omega)}^{p'\wedge2} +
\norm{\bs g(t)-\bs g_\infty}_{L^{r'}(\Gamma)}^{p'\wedge2},\ \text{with}\
s=\begin{cases}
2&\text{if } \frac65\leq p<2\\
q'&\text{if } p\geq2
\end{cases}
\end{equation}
and
\begin{equation*}
\zeta(t)=\|\bs a(x,t,\bs u)-\bs a_\infty(x,\bs u)\|_{\bs L^{p'}(\Omega\times\RR^3)}^{p'}.
\end{equation*}

\subsubsection{The degenerate case $p>2$}

\begin{theorem}\label{Maxwell:evo:assimp}
Let $p>2$, suppose that the operators $\bs a$ and $\bs a_\infty$ satisfy~{\em(\ref{Operador:a:prop}\,a, b, c')}
and
\begin{equation*}
\Tende{\displaystyle{\int_{\frac{t}2}^{t}}\big(\zeta(\tau)
+\xi(\tau)\big)\,d\tau}
 {t\rightarrow\infty} {0.}
\end{equation*}

Then we have
 \begin{equation*}
  \Tende{\norm{\bs h(t)-\bs h_\infty}_{\bs L^2(\Omega)}}{t\rightarrow\infty}{0}.
 \end{equation*}
\end{theorem}
\begin{proof}
Choosing for test function in (\ref{Maxwell:estacionario:variacional}), for a.e. $t\in\RR^+$,
$\bs w(t)=\bs h(t)-\bs h_\infty$, we have
\begin{equation}\label{Maxwell:evo:assimp:3}
 \int_\Omega \bs a_\infty(x,\Rot\bs h_\infty)\cdot\Rot\bs w(t)=\int_\Omega\bs f_\infty\cdot\bs w(t)+\int_\Gamma\bs g_\infty\cdot\bs w(t).
\end{equation}

Taking $\bs w(t)$ as test function in~(\ref{Maxwell:evolutivo:variacional}), for a.e.
$t\in\RR^+$,
\begin{equation*}
 \int_\Omega \partial_t\bs h(t)\cdot\bs w(t)+\int_\Omega \bs a\bigl(x,t,\Rot\bs h(t)\bigr)\cdot\Rot\bs w(t)=\int_\Omega\bs f(t)\cdot\bs w(t)+
 \int_\Gamma\bs g(t)\cdot\bs w(t),
\end{equation*}
we conclude that
\begin{multline}\label{Maxwell:evo:assimp:5}
 \int_\Omega\partial_t\bs h(t)\cdot\bs w(t)+
 \int_\Omega \Bigl(\bs a\bigl(x,t,\Rot\bs h(t)\bigr)-\bs a\bigl(x,t,\Rot\bs h_\infty\bigr)\Bigr)\cdot\Rot\bs w(t)\\
 =\int_\Omega\bigl(\bs f(t)-\bs f_\infty\bigr)\cdot\bs w(t)+
 \int_\Gamma\bigl(\bs g(t)-\bs g_\infty\bigr)\cdot\bs w(t)\\
 +\int_\Omega\Big(\bs a_\infty\big(x,\Rot\bs h_\infty\big)-\bs a\big(x,t,\Rot\bs h_\infty\big)\Big)\cdot\Rot\bs w(t).
\end{multline}

Since, by (\ref{Operador:a:prop}c'),
\begin{equation*}
\int_\Omega \Bigl(\bs a\bigl(x,t,\Rot\bs h(t)\bigr)-\bs a\bigl(x,t,\Rot\bs h_\infty\bigr)\Bigr)\cdot\Rot\bs w(t)\ge a_*\int_\Omega|\Rot\bs w(t)|^p,
\end{equation*}
subtracting (\ref{Maxwell:evo:assimp:3}) from (\ref{Maxwell:evo:assimp:5}),
using H\"older and Young inequalities and the Remark~\ref{PoincareTrace}, we have

\begin{multline}\label{Maxwell:evo:assimp:9}
 \frac12\frac{d}{dt}\int_\Omega|\bs w(t)|^2 +
  C\left(\int_\Omega|\bs w(t)|^2\right)^{\frac p2}\\
\leq D_1\left(\norm{\bs f(t)-\bs f_\infty}^{p'}_{\bs L^{q'}(\Omega)}+
  \norm{\bs g(t)-\bs g_\infty}^{p'}_{\bs L^{r'}(\Gamma)}\right)+D_2\,\zeta(t).
\end{multline}

Denoting
$$\phi(t)=\int_\Omega|\bs w(t)|^2\quad \text{and}\quad
l(t)=
2D_1\xi(t)+2D_2\,\zeta(t),$$
the inequality (\ref{Maxwell:evo:assimp:9}) is written as follows

\begin{equation*}
 \phi'(t)+2C\,\phi(t)^{\frac{p}{2}} \leq l(t).
\end{equation*}

So, applying Lemma \ref{lema:Simon} bellow with $t_0=\frac{t}{2}$, the theorem follows from
\begin{equation*}
 \int_\Omega |\bs h(t)-\bs h_\infty|^2
 \leq \left(\tfrac{C(p-2)}{2}t\right)^\frac{-2}{p-2} +
 \int_{\frac{t}{2}}^t l(\sigma)\,d\sigma.
\end{equation*}
\end{proof}

\begin{lemma}[\cite{Simon1975}, p~600]\label{lema:Simon}
Let $\phi$ be a real, continuous, positive function, a.e. differentiable in an interval $I\subseteq\RR$,
such that
 \begin{equation*}  
  \phi'(t)+c(t)\,\phi(t)^\frac p2\leq l(t)\quad\text{for a.e. }t\in I,
 \end{equation*}
 being $p>2$, $c\geq0$ and $l$ integrable in $I$. Then
 \begin{equation*}
 \forall\, t_0,t\in\, I:\, t_0\le t \quad \phi(t)\leq
 \left(\tfrac{p-2}{2}\int_{t_0}^tc(\sigma)\,d\sigma\right)^\frac{-2}{p-2} +
 \int_{t_0}^t l(\sigma)\,d\sigma.
 \end{equation*}
\end{lemma}

\subsubsection{The case $p=2$}

\begin{theorem}
Let $p=2$ and suppose that the operators $\bs a$ and $\bs a_\infty$ verify~{\em(\ref{Operador:a:prop}\,a, b, c')}
and
\begin{equation*}
\Tende{\displaystyle{\int_{t}^{t+1}}\big(\zeta(\tau)+\xi(\tau)\big)\,d\tau}
 {t\rightarrow\infty} {0.}
\end{equation*}

Then we have
 \begin{equation*}
  \Tende{\norm{\bs h(t)-\bs h_\infty}_{\bs L^2(\Omega)}}{t\rightarrow\infty}{0}.
 \end{equation*}
\end{theorem}
\begin{proof}
Arguing as in the previous theorem, calling $\bs w(t)=\bs h(t)-\bs h_\infty$, we  get
\begin{multline*}
\frac12\frac{d}{dt}\int_\Omega|\bs w(t)|^2+a_*\int_\Omega|\Rot\bs w(t)|^2
\le\int_\Omega(\bs f(t)-\bs f_\infty)\cdot\bs w(t)+\int_\Gamma(\bs g(t)-\bs g_\infty)\cdot\bs w(t)\\
+\int_\Omega\bigl(\bs a_\infty(x,\Rot\bs h_\infty)-\bs a(x,t,\Rot\bs h_\infty)\bigr)\cdot\Rot\bs w(t),
\end{multline*}
from which we obtain, using H\"older and Young inequalities and the Remark~\ref{PoincareTrace},

\begin{multline*}
\frac12\frac{d}{dt}\int_\Omega\big|\bs w(t)\big|^2+C_1\int_\Omega\big|\bs w(t)\big|^2\\
\le D_1\left(\|\bs f(t)-\bs f_\infty\|_{\bs L^{q'}(\Omega)}^2+\|\bs g(t)-\bs g_\infty\|_{\bs L^{r'}(\Gamma)}^2\right)+D_2\,\zeta(t).
\end{multline*}
So
\begin{equation}\label{2.47}
\frac{d\ }{dt}\int_\Omega\big|\bs w(t)\big|^2+C\int_\Omega\big|\bs w(t)\big|^2\le
l(t)\le l_0,
\end{equation}
where $C=2C_1$,
\begin{equation*}
l(t)=2D_1\xi(t)+2D_2\,\zeta(t),
\end{equation*}
and $l_0$ is a constant which exists by the assumptions on
$\bs a$, $\bs a_\infty$, $\bs f$, $\bs f_\infty$, $\bs g$ and $\bs g_\infty$.

In order to prove that
$\bs w\in L^\infty(0,\infty;\bs L^2(\Omega))$,
we multiply (\ref{2.47}) by $e^{Ct}$ and integrate in time, between
$\sigma$ and $\tau$, $\sigma\le\tau$. Then
\begin{equation}\label{2.48}
 \int_\sigma^\tau\int_\Omega
e^{Ct}\partial_t\big|\bs w(t)\big|^2+C\int_\sigma^\tau\int_\Omega
e^{Ct}\big|\bs w(t)\big|^2\le \int_\sigma^\tau l_0e^{Ct}.
\end{equation}

But
\begin{equation}\label{2.49}
 \int_\sigma^\tau\int_\Omega
e^{Ct}\partial_t\big|\bs w(t)\big|^2=e^{C\tau}\int_\Omega\big|\bs w(\tau)\big|^2-
e^{C\sigma}\int_\Omega\big|\bs w(\sigma)\big|^2-C
\int_\sigma^\tau\int_\Omega e^{Ct}\big|\bs w(t)\big|^2.
\end{equation}

Combining (\ref{2.48}) and (\ref{2.49}) we get
\begin{equation*}
e^{C\tau}\int_\Omega\big|\bs w(\tau)\big|^2
\le\tfrac{l_0}C\left(e^{C\tau}-e^{C\sigma}\right)+
e^{C\sigma}\int_\Omega\big|\bs w(\sigma)\big|^2
\end{equation*}
and taking $\tau=t$ and $\sigma=0$, there exists a positive constant $l_1$ such that,
for all $t$,
\begin{equation*}
\int_\Omega\big|\bs w(t)\big|^2\le\tfrac{l_0}C+\int_\Omega\big|\bs h_0-\bs h_\infty\big|^2\le l_1.
\end{equation*}

Applying Lemma \ref{lema:Haraux} below, fixing $t_0>0$, for all $t>t_0$ we have
\begin{equation*}
\int_\Omega\big|\bs h(t)-\bs h_\infty\big|^2\le
e^{C(t_0-t)}l_1+\tfrac1{1-e^{-C}}\sup_{\tau\ge
t_0}\int_\tau^{\tau+1}l(\sigma)\,d\sigma.
\end{equation*}
\end{proof}

\begin{lemma}[\cite{Haraux1981}, p~286]\label{lema:Haraux}
Let $\phi(t)$ be a nonnegative function, absolutely continuous in any compact interval of $\RR^+$,
$l(t)$ a nonnegative function belonging to  $L^1_{\rm loc}(\RR^+)$ and $c$ a positive function such that
\begin{equation*} 
  \phi'(t)+c\,\phi(t)\leq l(t),\quad \forall t\,\geq 0.
\end{equation*}

Then
\begin{equation*}
\forall\, t_0,t\in\, \RR^+:\, t_0\le t\quad \phi(t)\le
e^{c(t_0-t)}\phi(t_0)+\tfrac1{1-e^{-c}}\sup_{\tau\ge
t_0}\int_{\tau}^{\tau+1} l(\sigma)d\sigma.
\end{equation*}
\end{lemma}

\subsubsection{The singular case for $\frac65\leq p<2$ and $\bs a(x,t,\bs u)=\bs a_\infty(x,\bs u)=\nu(x) |\bs u|^{p-2}\bs u$}

\begin{theorem}
Let {\em(\ref{a-potential})} hold, $\frac65\leq p<2$ and
$\partial_t \bs g\in L^{\infty}(0,\infty;\bs L^{r'}(\Gamma))$.
Suppose that
\begin{equation*}
\Tende{\displaystyle{\int_t^{t+1}\xi(\tau)\,d\tau} }{t\rightarrow\infty} {0.}
\end{equation*}

Then we have
 \begin{equation*}
  \Tende{\norm{\bs h(t)-\bs h_\infty}_{\bs L^2(\Omega)}}{t\rightarrow\infty}{0}.
 \end{equation*}
\end{theorem}
\begin{proof}
By the property (\ref{Operador:a:prop} c'),
\begin{multline*}
\nu\bigl(|\Rot\bs h(t)|^{p-2}\Rot\bs h(t)-|\Rot\bs h_\infty|^{p-2}\Rot\bs h_\infty\bigr)\cdot\bigl(\Rot\bs h(t)-\Rot\bs h_\infty\bigr)\\
\ge a_*\bigl(|\Rot\bs h(t)|+|\Rot\bs h_\infty|\bigr)^{p-2}
|\Rot(\bs h(t)-\bs h_\infty)|^2.
\end{multline*}

Setting $\bs w(t)=\bs h(t)-\bs h_\infty$, recalling (\ref{Maxwell:evo:assimp:5}) and using the above inequality, we obtain
\begin{multline*}
\frac{d\ }{dt}\int_\Omega|\bs w(t)|^2+a_*\int_\Omega\big(|\Rot\bs h(t)|+|\Rot\bs h_\infty|\big)^{p-2}| \Rot\bs w(t)|^2\\
\le \int_\Omega\bigl(\bs f(t)-\bs f_\infty\bigr)\cdot\bs w(t)+
 \int_\Gamma\bigl(\bs g(t)-\bs g_\infty\bigr)\cdot\bs w(t).
\end{multline*}

We recall now the inverse H\"{o}lder inequality (see \cite{Sobolev1963}, p~8): let $0<s<1$ and $s'=\frac{s}{s-1}$. If $F\in L^s(\Omega)$, $FG\in L^1(\Omega)$ and $\displaystyle\int_\Omega
 |G(x)|^{s'}dx<\infty$ then
\begin{equation*}
\left(\int_\Omega|F(x)|^sdx\right)^{\frac1{s}}\le
\int_\Omega|F(x)G(x)|dx\left(\int_{\Omega}|G(x)|^{s'}dx\right)^{\frac{-1}{s'}}
\end{equation*}
and we apply it, with $s=\frac{p}2$ (so $s'=\frac{p}{p-2}$),
$F=|\Rot(\bs h(t)-\bs h_\infty)|^2$ and $G=\big(|\Rot\bs h(t)|+|\Rot\bs h_\infty|\big)^{p-2}$, in $\widehat{\Omega}
 =\{x\in\Omega:|\Rot\bs h(x,t)|+|\Rot\bs h_\infty(x)|\neq 0\}$.

So,
\begin{multline*}
\Bigl(\int_{\widehat{\Omega}}\bigl(|\Rot(\bs h(t)-
\bs h_\infty)|^2\bigr)^{\frac{p}2}\Bigr)^{\frac2{p}}\\
 \le
\int_{\widehat{\Omega}}|\Rot(\bs h(t)-\bs h_\infty)|^2\big(|\Rot\bs h(t)|+|\Rot\bs h_\infty|\big)^{p-2}\,
\Big(\int_{\widehat{\Omega}}\big(|\Rot\bs h(t)|+|\Rot\bs h_\infty|\big)^p\Big)^{\frac{2-p}p}
\end{multline*}
and
\begin{multline}\label{rotp}
\int_{\widehat{\Omega}}|\Rot(\bs h(t)-\bs h_\infty)|^2\big(|\Rot\bs h(t)|+|\Rot\bs h_\infty|\big)^{p-2}\\
\ge\Bigl(\int_{\widehat{\Omega}}|\Rot(\bs h(t)-\bs h_\infty)|^p\Bigr)^{\frac2{p}}\,\Bigl(\int_{\widehat{\Omega}}\big(|\Rot\bs h(t)|+|\Rot\bs h_\infty|\big)^p\Bigr)^{\frac{p-2}p}.
\end{multline}

From (\ref{Maxwell:estacionario:variacional}) and the assumptions we have
\begin{equation*}
\int_\Omega|\Rot\bs h_\infty|^p\le C_1.
\end{equation*}

Simple calculations allows us to rewrite the inequality (\ref{te_proposition_estimate})
in the form
\begin{multline*}
\norm{\partial_t\bs h}^2_{\bs L^2(\Omega\times(0,\infty))} +
\norm{\Rot\bs h}^p_{L^\infty(0,\infty;\bs L^p(\Omega))}\\
\leq
C_2\Big(\|\bs f\|^2_{L^\infty(0,\infty;\bs L^2(\Omega))} +
\|\bs g\|^{p'}_{L^\infty(0,\infty;\bs L^{r'}(\Gamma))}
+\|\partial_t\bs g\|^{p'}_{L^{\infty}(0,\infty;\bs L^{r'}(\Gamma))}\Big)+C_3,
\end{multline*}
where $C_2$ and $C_3$ are positive constants.

We get, using the Proposition~\ref{te_proposition},
\[
 \Big(\int_{\Omega}\big(|\Rot\bs h(t)|+|\Rot\bs h_\infty|\big)^p\Big)^{\frac{2-p}p}\leq C_4,
\]
and, from (\ref{rotp}),
\begin{equation*}
\int_{\hat{\Omega}}\bigl|\Rot(\bs h(t)-\bs h_\infty)\bigr|^2\big(|\Rot\bs h(t)|+|\Rot\bs h_\infty|\big)^{p-2}
\ge\Bigl(\int_{\hat{\Omega}}\bigl|\Rot(\bs h(t)-\bs h_\infty)\bigr|^p\Bigr)^{\frac2{p}}\frac1{C_4}.
\end{equation*}

By the Remark~\ref{PoincareTrace} we know, since $p\geq\frac65$, that
$$\frac12\frac{d\ }{dt}\int_\Omega|\bs w(t)|^2+C_5\int_\Omega|\bs w(t)|^2
\le D_1\left(\|\bs f(t)-\bs f_\infty\|_{\bs L^{2}(\Omega)}^2+\|\bs g(t)-\bs g_\infty\|_{\bs L^{r'}(\Gamma)}^2\right),$$
and so, for $C=2C_5$ and $l(t)=2 D_1\xi(t)$
we deduce that
 $$\frac{d\ }{dt}\int_\Omega|\bs w(t)|^2+C\int_\Omega|\bs w(t)|^2\le l(t),$$
and the proof is concluded exactly as the previous one.
 \end{proof}

\section{A limit problem when $n\rightarrow\infty$}\label{sec3}

Given $p>1$ let $\bs\delta:Q_T\times\RR^3\rightarrow\RR^3$ be a
Carath\'{e}odory function satisfying (\ref{Operador:a:prop:b}),
the monotonicity condition
\begin{equation}\label{delta1}
\big(\bs\delta(x,t,\bs u)-\bs\delta(x,t,\bs v)\big)\cdot (\bs u-\bs v)\ge 0,
\end{equation}
and also
\begin{equation}\label{delta2}
\bs\delta\in\bs L^1(Q_T\times\RR^3),\quad \partial_t\bs\delta\in\bs L^1(Q_T\times\RR^3)\quad\text{and}\quad\bs\delta(x,t,0)=0.
\end{equation}

Let
\begin{equation*}
\mathbb K_*=\big\{\bs v\in\, \W{p}{\Omega}:|\Rot\bs v|\le 1\mbox{ a.e. in }\Omega\big\}
\end{equation*}
and assume that
\begin{equation}\label{fgh}
\bs f \in \bs L^{2}(Q_T),\quad
\bs g \in L^\infty(0,T;\bs L^{1}(\Gamma))\cap W^{1,1}(0,T;\bs L^1(\Gamma))
\quad\text{and}\quad
\bs h_0 \in\K_*.
\end{equation}
For $n\in\mathbb N$, $n>3\vee p$, define
\begin{equation}\label{an}
\bs a_n(x,t,\bs u)=|\Rot \bs u|^{n-2}\Rot \bs u+\bs\delta(x,t,\bs u)
\end{equation}
and consider the following problem: to find $\bs h_n\in L^n(0,T;\W{n}{\Omega})\cap C([0,T];\bs L^2_\sigma(\Omega))$ and
$\partial_t\bs h_n \in L^{n'}(0,T;\W{n}{\Omega}')$ such that, for a.e.\ $t\in(0,T)$,
\begin{gather}\label{Maxwell:evolutivo:variacional:n}
\begin{split}
\int_\Omega\partial_t\bs h_n(t)\cdot\bs\varphi + \int_\Omega |\Rot \bs h_n(t)|^{n-2}&\Rot \bs h_n(t)\cdot\Rot\bs\varphi\\
+\int_\Omega\bs\delta(x,t,\Rot\bs h_n(t))\cdot\Rot\bs\varphi
&=\int_\Omega\bs f(t)\cdot\bs\varphi +\int_\Gamma\bs g(t)\cdot\bs\varphi\quad\forall\bs\varphi\in\W{n}{\Omega},\\
\bs h_n(0)&=\bs h_0.
\end{split}
\end{gather}

We define the variational inequality: to find $\bs h_*\in H^1(0,T;\bs L^2(\Omega))\cap L^2(0,T;\W{\infty}{\Omega})$ such that
 for a.e. $t\in\,(0,T)$, $\bs h_*(t)\in\,\mathbb K_*$,
\begin{gather}\label{iv*}
\begin{split}
\int_\Omega\partial_t \bs h_*(t)\cdot(\bs v-&\bs h_*(t))+\int_\Omega\bs\delta(x,t,\Rot\bs h_*(t))\cdot\Rot(\bs v-\bs h_*(t))\\
&\ge \int_\Omega\, \bs f(t)\cdot(\bs v-\bs h_*(t))+\int_\Gamma\, \bs g(t)\cdot(\bs v-\bs h_*(t))\quad\forall \bs v\in \mathbb K_*,\\
\bs h_*(0)&=\bs h_0.
\end{split}
\end{gather}

\begin{remark}
Note that (\ref{iv*}) has at most one solution and observe that the operator $\bs\delta$ may be the null operator.
\hfill{$\square$}\end{remark}

\begin{proposition}
With the assumptions  {\em (\ref{delta1})}, {\em (\ref{delta2})} and {\em (\ref{fgh})}, let $\bs h_n$ be the solution
of the problem {\em (\ref{Maxwell:evolutivo:variacional:n})}.

Then there exists a positive constant $C$, independent of $n$, such that
\begin{equation}\label{estimates}
\|\bs h_n\|_{L^\infty(0,T;\bs L^2(\Omega))}\le C,\quad \|\Rot\bs h_n\|_{\bs L^n(Q_T)}\le C,\quad
\|\partial_t \bs h_n\|_{\bs L^{2}(Q_T)}\le C.
\end{equation}
\end{proposition}
\begin{proof} Choosing $\bs h_n$ as test function in (\ref{Maxwell:evolutivo:variacional:n}) and using the Remark~\ref{PoincareTrace}, we obtain
\begin{alignat*}{1}
\tfrac12\int_\Omega|\bs h_n(t)|^2&+\int_0^t\int_\Omega|\Rot\bs h_n|^n+\int_0^t\int_\Omega\bs\delta(x,t,\Rot\bs h_n)\cdot\Rot\bs h_n\\
&=\int_0^t\int_\Omega\bs f\cdot\bs h_n+\int_0^t\int_\Gamma\bs g\cdot\bs h_n+\tfrac12\int_\Omega|\bs h_0|^2\\
&\le \tfrac{C_1}{n'}\left(\int_0^t\int_\Omega|\bs f|+\int_0^t\int_\Gamma|\bs g|\right)+\tfrac1n\int_0^t\int_\Omega|\Rot\bs h_n|^n+\tfrac12\int_\Omega|\bs h_0|^2,
\end{alignat*}
where $C_1$ is a positive constant independent of $n$.

Since, for a.e.\ $t\in(0,T)$,
\begin{equation*}
\bs\delta(x,t,\Rot\bs h_n(t))\cdot\Rot\bs h_n(t)\ge 0,
\end{equation*}
the first two inequalities of (\ref{estimates}) follow immediately.

On the other hand, formally we have from (\ref{Maxwell:evolutivo:variacional:n}), with $\bs\varphi=\partial_t\bs h_n(t)$,
\begin{multline*}
\int_\Omega|\partial_t \bs h_n(t)|^2+\int_\Omega |\Rot \bs h_n(t)|^{n-2}\Rot \bs h_n(t)\cdot\partial_t\Rot\bs h_n(t)\\
+\int_\Omega \bs\delta(x,t,\bs h_n(t))\cdot\partial_t\Rot\bs h_n(t)
=\int_\Omega\, \bs f(t)\cdot\partial_t \bs h_n(t)+\int_\Gamma\, \bs g(t)\cdot\partial_t \bs h_n(t).
\end{multline*}

Since
\begin{multline*}
\int_0^t\int_\Omega\bs\delta(x,t,\Rot\bs h_n)\cdot\partial_t\Rot \bs h_n
=\int_\Omega \bs\delta(x,t,\Rot\bs h_n(t))\cdot\Rot \bs h_n(t)\\
-\int_\Omega \bs\delta(x,0,\Rot\bs h_n(0))\cdot\Rot\bs h_n(0)-\int_0^t\int_\Omega \partial_t \bs\delta(x,t,\Rot\bs h_n)\cdot \bs h_n\\
\geq -\int_\Omega \bs\delta(x,0,\Rot\bs h_0)\cdot\Rot\bs h_0-\int_0^t\int_\Omega \partial_t \bs\delta(x,t,\Rot\bs h_n)\cdot \bs h_n,
\end{multline*}
and, on the other hand,
\begin{equation}\label{gdth}
\int_0^t\int_\Gamma \bs g\cdot\partial_t \bs h_n=\int_\Gamma  \bs g(t)\cdot\bs h_n(t)-\int_\Gamma\bs g(0)\cdot\bs h_0-
\int_0^t\int_\Gamma\partial_t\bs g\cdot\bs h_n,
\end{equation}
\begin{equation}\label{gdth2}
\int_\Gamma  \bs g(t)\cdot\bs h_n(t)\le C_2\Big(\int_\Gamma|\bs g(t)|\Big)^{n'}+\tfrac1n\int_\Omega|\Rot\bs h_n(t)|^n.
\end{equation}
we have
\begin{multline*}
\tfrac12\int_0^t\int_\Omega\left|\partial_t \bs h_n\right|^2
\le
\tfrac1n\|\Rot\bs h_0\|^n_{\bs L^n(\Omega)}+
\norm{\bs \delta(\cdot,0,\cdot)}_{\bs L^1(\Omega\times\RR^3)}\norm{\Rot\bs h_0}_{\bs L^\infty(\Omega)}\\
+\|\partial_t \bs\delta\|_{\bs L^1(Q_T\times\RR^3)}\|\bs h_n\|_{\bs L^\infty(Q_T)}
+
\tfrac12\|\bs f\|^2_{\bs L^{2}(Q_T)}\\
+C_2\|\bs g\|^{n'}_{L^\infty(0,T;\bs L^1(\Gamma))}
+\norm{\bs g(0)}_{\bs L^1(\Gamma)}\norm{\bs h_0}_{\bs L^\infty(\Gamma)}
+\norm{\partial_t\bs g}_{\bs L^1(Q_T)}\norm{\bs h_n}_{\bs L^\infty(Q_T)}
\end{multline*}
and so $\{\partial_t\bs h_n\}_n$ is uniformly bounded in $\bs L^2(Q_T)$.
\end{proof}

\begin{theorem}
Let $\bs h_n$ be the solution of the problem {\em (\ref{Maxwell:evolutivo:variacional:n})}.
Then, at least for subsequences, we have
\begin{alignat*}{2}
\bs h_n\lra&\ \bs h_*&&\mbox{ in }C([0,T];\bs L^2(\Omega))\mbox{-strong},\\
\partial_t\bs h_n\lraup&\ \partial_t\bs h_* &&\mbox{ in }\bs L^2(Q_T)\mbox{-weak},\\
\Rot\bs h_n\lraup&\ \Rot\bs h_*&\quad &\mbox{ in }\bs L^q(Q_T)\mbox{-weak},
\end{alignat*}
for any fixed $3<q<\infty$, where $\bs h_*$ is the solution of the problem {\em (\ref{iv*})}.
\end{theorem}
\begin{proof}
By the uniform estimates in (\ref{estimates}) we only need to check that $\bs h_*$ solves (\ref{iv*}).

Let $\bs\varphi\in\W{p}{\Omega}$ be such that $|\Rot\bs\varphi|<1$ a.e.. Taking $\bs\varphi-\bs h_n$ as a test function in (\ref{Maxwell:evolutivo:variacional:n}), we have
\begin{multline*}
\int_{Q_T}\partial_t\bs h_n\cdot(\bs\varphi-\bs h_n)+\int_{Q_T}|\Rot\bs h_n|^{n-2}\Rot\bs h_n\cdot\Rot(\bs\varphi-\bs h_n)\\
+\int_{Q_T}\bs\delta(x,t,\Rot\bs h_n)\cdot\Rot(\bs\varphi-\bs h_n)=\int_{Q_T}\bs f\cdot(\bs\varphi-\bs h_n)+\int_{\Sigma_T}\bs g\cdot(\bs\varphi-\bs h_n).
\end{multline*}

By the monotonicity of the operator $\bs a_n$ defined in (\ref{an}) we have
\begin{multline}\label{liminf1}
\int_{Q_T}\partial_t\bs h_n\cdot(\bs\varphi-\bs h_n)+\int_{Q_T}|\Rot\bs \varphi|^{n-2}\Rot\bs \varphi\cdot\Rot(\bs\varphi-\bs h_n)\\
+\int_{Q_T}\bs\delta(x,t,\Rot\bs \varphi)\cdot\Rot(\bs\varphi-\bs h_n)\ge\int_{Q_T}\bs f\cdot(\bs\varphi-\bs h_n)+\int_{\Sigma_T}\bs g\cdot(\bs\varphi-\bs h_n).
\end{multline}

Applying limit in $n$ to both members of (\ref{liminf1}) we get
\begin{multline}\label{h*}
\int_{Q_T}\partial_t\bs h_*\cdot(\bs\varphi-\bs h_*)+
\int_{Q_T}\bs\delta(x,t,\Rot\bs \varphi)\cdot\Rot(\bs\varphi-\bs h_*)\\
\ge\int_{Q_T}\bs f\cdot(\bs\varphi-\bs h_*)+\int_{\Sigma_T}\bs g\cdot(\bs\varphi-\bs h_*).
\end{multline}

Since $\bs\varphi$ is an arbitrary function of $\W{p}{\Omega}$ satisfying $|\Rot\bs\varphi|<1$, the  inequality (\ref{h*}) still holds,
by density, for all $\bs\varphi\in\mathbb K_*$. We also have $\bs h_*(0)=\bs h_0$.

Given $p<q<n$, by (\ref{estimates}),
\begin{equation*}
\norm{\Rot\bs h_n}_{\bs L^q(Q_T)}
\le |Q_T|^{\frac1q-\frac1n}\norm{\Rot\bs h_n}_{\bs L^n(Q_T)}\le  |Q_T|^{\frac1q-\frac1n}C^{\frac1n}
\end{equation*}
and
\begin{equation*}
\norm{\Rot \bs h_*}_{\bs L^q(Q_T)}
\le\liminf_n\|\Rot\bs h_n\|_{\bs L^q(Q_T)}\le |Q_T|^{\frac1q}\qquad\forall\, q>p,
\end{equation*}
so $\Rot \bs h_*\in\bs L^\infty(Q_T)$ and
$\|\Rot \bs h_*\|_{\bs L^\infty(Q_T)}\le 1$ which proves that
 $\bs h_*(t)$ belongs to the convex set $\K_*$ for a.e.\ $t\in(0,T)$.

Choosing $\bs \varphi=\bs h_*+\lambda(\bs w-\bs h_*)$, with $\lambda\in\,(0,1]$ and $\bs w$ any element of $\mathbb K_*$, we have
\begin{multline*}
\int_{Q_T}\partial_t\bs h_*\cdot(\bs w-\bs h_*)
+\int_{Q_T}\bs\delta(x,t,\Rot\bs h_*+\lambda(\bs w-\bs h_*))\cdot\Rot(\bs w-\bs h_*)\\
\ge\int_{Q_T}\bs f\cdot(\bs w-\bs h_*)+\int_{\Sigma_T}\bs g\cdot(\bs w-\bs h_*).
\end{multline*}

Letting $\lambda\rightarrow 0$, we get
\begin{multline*}
\int_{Q_T}\partial_t\bs h_*\cdot(\bs w-\bs h_*)
+\int_{Q_T}\bs\delta(x,t,\Rot\bs h_*)\cdot\Rot(\bs w-\bs h_*)\\
\ge\int_{Q_T}\bs f\cdot(\bs w-\bs h_*)+\int_{\Sigma_T}\bs g\cdot(\bs w-\bs h_*).
\end{multline*}

Standard arguments imply that, for a.e.\ $t\in(0,T)$,
\begin{multline*}
\int_\Omega\partial_t\bs h_*(t)\cdot(\bs w-\bs h_*(t))
+\int_\Omega\bs\delta(x,t,\Rot\bs h_*(t))\cdot\Rot(\bs w-\bs h_*(t))\\
\ge\int_\Omega\bs f\cdot(\bs w-\bs h_*(t))+\int_\Sigma\bs g\cdot(\bs w-\bs h_*(t)).
\end{multline*}
\end{proof}

\begin{remark}
If $\bs\delta=\bs\delta(x,\bs u)$ is independent of $t$ (in particular $\bs\delta=\bs0$) in the corresponding stationary problem
(\ref{Maxwell:evolutivo:variacional:n}) with stationary data $\bs f(x,t)=\bs f_\infty(x)$ and $\bs g(x,t)=\bs g_\infty(x)$, i.e.
\begin{multline*}
\int_\Omega |\Rot \bs h_{n\infty}|^{n-2}\Rot \bs h_{n\infty}\cdot\Rot\bs\varphi\\
+\int_\Omega\bs\delta(x,\Rot\bs h_{n\infty})\cdot\Rot\bs\varphi
=\int_\Omega\bs f_\infty\cdot\bs\varphi +\int_\Gamma\bs g_\infty\cdot\bs\varphi\quad\forall\bs\varphi\in\W{n}{\Omega},
\end{multline*}
it was shown in \cite{MirandaRodriguesSantos2009} that there exists subsequences $n'\rightarrow\infty$ and $\bs h_{*\infty}\in\K_*$
such that
\begin{equation}\label{hn'}
\bs h_{n'\infty}\lraup\bs h_{*\infty}\quad\text{in}\quad\W{q}{\Omega}\text{-weak},\quad \text{for}\quad n'\rightarrow\infty,
\end{equation}
for any fixed $3<q<\infty$,  where $h_{*\infty}$ is a solution in $\K_*$ of
\begin{multline}\label{deltah*}
\int_\Omega\bs\delta(x,\Rot\bs h_{*\infty})\cdot\Rot(\bs\varphi-\bs h_{*\infty})\\
\ge\int_\Omega\bs f_\infty\cdot(\bs\varphi-\bs h_{*\infty}) +
\int_\Gamma\bs g_\infty\cdot(\bs\varphi-\bs h_{*\infty})\quad\forall\bs\varphi\in\K_*.
\end{multline}

In general, (\ref{deltah*}) may have more than one solution if $\bs\delta$ is not strictly monotone,
in particular when $\bs\delta=\bs0$.
\hfill{$\square$}\end{remark}

\begin{remark}
If we apply Theorem~\ref{Maxwell:evo:assimp} for each fixed $n>3$ with
\begin{equation*}
\Tende{
\theta(t)=\displaystyle\int_{\frac{t}{2}}^t \big(\norm{\bs f(\tau)-\bs f_\infty}_{\bs L^1(\Omega)}+
\norm{\bs g(\tau)-\bs g_\infty}_{\bs L^1(\Gamma)}\big)\,d\tau
}{t\rightarrow\infty}{0}
\end{equation*}
we have the estimate
\begin{equation*}
\int_\Omega|\bs h_n(t)-\bs h_{n\infty}|^2\leq \big(C(n-2)t\big)^{\frac{-2}{n-2}}+ \theta(t),
\end{equation*}
where the constant $C>0$ is independent of $n$ and $t$.

So, for a subsequence $n'$ satisfying (\ref{hn'}), there exists a sequence, $t_{n'}\rightarrow\infty$, such that
\begin{equation*}
\bs h_{n'}(t_{n'})\lraup\bs h_{*\infty}\quad\text{in}\quad\bs L^{2}(\Omega)\text{-weak},\quad \text{for}\quad n'\rightarrow\infty.
\end{equation*}

An interesting open question in the degenerate case is whether there exists a sequence $t_{n}\rightarrow\infty$ such that
$\bs h_*(t_{n})$ converges, in some sense, to $\bs h_{*\infty}$.
\hfill{$\square$}\end{remark}

\section{The variational inequality with evolutionary curl constraint}\label{sec4}

Define, for a.e.\ $t\in(0,T)$, the following closed convex subset of $\W{p}{\Omega}$,
\begin{equation*}
 \K(t)=\big\{\bs v\in\W{p}{\Omega} : |\Rot\bs v|\leq\Psi(t),\text{ a.e. in }\Omega\big\},
\end{equation*}
where $\Psi:Q_T\longrightarrow\RR^+$ is a function such that $\Psi\ge\alpha>0$.

In this section we assume the following regularity of the data:
\begin{equation*}
\begin{split}
\bs f\in \bs L^{q'\vee2}(Q_T),\quad
\bs g\in L^\infty(0,T;\bs L^{r'}(\Gamma))\cap W^{1,r'}(0,T;\bs L^{r'}(\Gamma)),\\
\nu\in L^\infty(\Omega),\ 0<a_*\leq\nu\leq a^*,\quad\Psi \in W^{1,\infty}(0,T;L^\infty(\Omega))\quad \text{and}\quad
\bs h_0\in\K(0).
\end{split}
\end{equation*}

We define the variational inequality: to find $\bs h$, in a suitable class of functions,
such that
\begin{equation}\label{iv}
\begin{split}
\bs h(t)\in \K(t), \text{ for a.e. }t\in\,(0,T),\quad
\bs h(0)=\bs h_0,\\
\int_\Omega\partial_t \bs h(t)\cdot(\bs\varphi-\bs h(t))+
\int_\Omega\nu|\Rot\bs h(t)|^{p-2}\Rot\bs h(t)\cdot\Rot(\bs \varphi-\bs h(t))\\
\ge \int_\Omega\, \bs f(t)\cdot(\bs\varphi-\bs h(t))+
\int_\Gamma\, \bs g(t)\cdot(\bs\varphi-\bs h(t)),\\
\qquad\forall \bs \varphi\in \K(t),
\text{ for a.e. }t\in\,(0,T).
\end{split}
\end{equation}

\subsection{The approximated problem}

Following a natural constraint penalization also used in a similar scalar parabolic problem \cite{Santos1991,Santos2002},
we introduce a small positive parameter $\varepsilon<1$.

Let us consider a continuous bounded increasing function $k_\varepsilon:\RR\longrightarrow\RR^+$ , satisfying
\begin{equation*}
k_\varepsilon(s)=
\begin{cases}
 1,&\text{ if }s\leq 0,\\
 e^{\frac{s}{\varepsilon}},&\text{ if }\varepsilon\leq s\leq \frac{1}{\varepsilon}-\varepsilon,\\
 e^{\frac{1}{\varepsilon^2}},&\text{ if } s\geq \frac{1}{\varepsilon}.
\end{cases}
\end{equation*}
and define, for $(x,t)\in Q_T$ and $\bs u\in\RR^3$,
\begin{equation*}
\bs a(x,t,\bs u)=\nu(x) k_\varepsilon\big(|\bs u|^p-\Psi^p(x,t)\big)|\bs u|^{p-2}\bs u.
\end{equation*}

The operator $A$, as defined in (\ref{operador:A:def}) with this $\bs a$, is bounded, monotone, coercive and hemicontinuous and so, by Proposition~\ref{Maxwell:evolutivo}, for each $\varepsilon>0$,
the approximated problem
\begin{equation}\label{ProblemaAproximado}
\begin{split}
\int_\Omega\partial_t\bs h_\varepsilon(t) \cdot\bs\varphi\ +&
\int_\Omega \nu k_\varepsilon\big(|\Rot \bs h_\varepsilon(t)|^p-\Psi(t)^p\big)|\Rot \bs h_\varepsilon(t)|^{p-2}\Rot \bs h_\varepsilon(t)\cdot\Rot\bs\varphi\\
=&\int_\Omega\bs f(t)\cdot\bs\varphi
+ \int_\Gamma\bs g(t)\cdot\bs\varphi\quad\forall\bs\varphi\in\Wtilde{p}{\Omega},\ \text{for a.e. }t\in\,(0,T),\\
\bs h_\varepsilon(0)=&\ \bs h_0,
\end{split}
\end{equation}
has a unique solution, $\bs h_\varepsilon\in L^p(0,T;\Wtilde{p}{\Omega})\cap C([0,T];\bs L^2_\sigma(\Omega))$, satisfying the estimate~(\ref{estimate}), independently of $\varepsilon$.
Since $k_\varepsilon(s)\geq1$, we have the following lemma.

\begin{lemma}\label{PApEstRoth}
There is a positive constant $C$ such that, for all $0<\varepsilon<1$,
\begin{equation*}
\norm{\bs h_\varepsilon}_{L^\infty(0,T;\bs L^2(\Omega))}+\norm{\Rot{\bs h_\varepsilon}}_{\bs L^p(Q_T)}\leq C.
\end{equation*}
\end{lemma}

\subsection{Existence of solution of the variational inequality}
In order to prove that a subsequence of the solutions of the approximate problems converges, with $\epsilon\rightarrow0$, for the solution of the variational inequality, we need additional a priori estimates.

\begin{lemma}\label{PApEstk}
There is a positive constant $C$ such that, for $0<\varepsilon<1$,
\begin{equation*}
\norm{k_\varepsilon(|\Rot{\bs h_\varepsilon}|^p- \Psi^p)}_{L^1(Q_T)}\leq C.
\end{equation*}
\end{lemma}

\begin{proof}
Choosing in (\ref{ProblemaAproximado}) $\bs\varphi=\bs h_\varepsilon$, we obtain, for a positive constant $C_1$,
\begin{multline}\label{PAEstRoth2}
a_*\int_{Q_T}k_\varepsilon(|\Rot\bs h_\varepsilon|^p-\Psi^p)|\Rot\bs h_\varepsilon|^p\\
\leq C_1\left(\norm{\bs f}_{\bs L^{q'}(Q_T)}^{p'}+
\norm{\bs g}_{\bs L^{r'}(\Sigma_T)}^{p'}+
\norm{\bs h_0}^2_{\bs L^{2}(\Omega)} \right).
\end{multline}
Observing that $k_\varepsilon(s)=1$ for $s\leq0$ and $k_\varepsilon(s)s\geq0$ for $s\geq0$,
we have
\begin{equation}\label{PAEstk2}
\begin{split}
\int_{Q_T}k_\varepsilon(|\Rot&\bs h_\varepsilon|^p-\Psi^p)(|\Rot\bs h_\varepsilon|^p-\Psi^p)\\
&=\int_{\{|\Rot\bs h_\varepsilon|^p-\Psi^p\leq0\}}\hspace{-15mm}k_\varepsilon(|\Rot\bs h_\varepsilon|^p-\Psi^p)(|\Rot\bs h_\varepsilon|^p-\Psi^p)+\\
&\qquad\qquad\int_{\{|\Rot\bs h_\varepsilon|^p-\Psi^p>0\}}\hspace{-15mm}k_\varepsilon(|\Rot\bs h_\varepsilon|^p-\Psi^p)(|\Rot\bs h_\varepsilon|^p-\Psi^p)\\
&\geq \int_{\{|\Rot\bs h_\varepsilon|^p-\Psi^p\leq0\}}\hspace{-15mm}|\Rot\bs h_\varepsilon|^p-
\int_{\{|\Rot\bs h_\varepsilon|^p-\Psi^p\leq0\}}\hspace{-15mm}\Psi^p\\
&\geq -\int_{Q_T}\Psi^p.
\end{split}
\end{equation}

Recalling that $\Psi\geq\alpha> 0$ we obtain
\begin{align*}
\int_{Q_T}k_\varepsilon (|\Rot\bs h_\varepsilon|^p-\Psi^p)&\leq
\int_{Q_T}k_\varepsilon (|\Rot\bs h_\varepsilon|^p-\Psi^p)\frac{\Psi^p}{\alpha^p}\\
& =\tfrac{1}{\alpha^p}\int_{Q_T}k_\varepsilon(|\Rot\bs h_\varepsilon|^p-\Psi^p)%
(\Psi^p-|\Rot\bs h_\varepsilon|^p)\\
&\qquad\qquad+\tfrac{1}{\alpha^p}\int_{Q_T}k_\varepsilon(|\Rot\bs h_\varepsilon|^p-\Psi^p)%
|\Rot\bs h_\varepsilon|^p.
\end{align*}

Applying the relations (\ref{PAEstRoth2}) and (\ref{PAEstk2}) to the last inequality we have
\begin{multline*}
\int_{Q_T}k_\varepsilon (|\Rot\bs h_\varepsilon|^p-\Psi^p)\\
\leq
\tfrac{C_1}{a_*\alpha^p}\left(\norm{\bs f}_{\bs L^{q'}(Q_T)}^{p'}+
\norm{\bs g}_{\bs L^{r'}(\Sigma_T)}^{p'}+
\norm{\bs h_0}_{\bs L^{2}(\Omega)}\right) +\tfrac{1}{\alpha^p}\norm{\Psi^p}_{L^p(\Omega)}^p.
\end{multline*}
\end{proof}

\begin{lemma}\label{PAEstdth}
There is a positive constant $C$ such that, for $0<\varepsilon<1$,
\begin{equation*}
\norm{\partial_t\bs h_\varepsilon}_{\bs L^2(Q_T)}\leq C.
\end{equation*}
\end{lemma}

\begin{proof}
Using Galerkin approximations, we can use $\partial_t\bs h_\varepsilon(t)$ formally as test function in
equation (\ref{ProblemaAproximado}). Integrating that equation over $(0,t)$ and setting
\begin{equation*}
\phi_\varepsilon(s)=\int_0^s k_\varepsilon(\tau)\,d\tau,
\end{equation*}
we have
\begin{multline}\label{PAEstdth1}
\int_{Q_t}|\partial_t\bs h_\varepsilon|^2 +
\tfrac1p\int_{Q_t}\nu\, \partial_t\Big(\phi_\varepsilon\big(|\Rot \bs h_\varepsilon|^p-\Psi^p\big)\Big)\\
+\int_{Q_t}\nu k_\varepsilon (|\Rot\bs h_\varepsilon|^p-\Psi^p)\Psi^{p-1}\partial_t\Psi
=\int_{Q_t}\bs f\cdot\partial_t\bs h_\varepsilon
+ \int_{\Sigma_t}\bs g\cdot\partial_t\bs h_\varepsilon.
\end{multline}

Observing that $\phi_\varepsilon(s)=s$, if $s\leq0$, and $\phi_\varepsilon(s)\geq s$, if $s\geq0$,

\begin{multline}\label{PAEstdth2}
 \tfrac1p\int_{Q_t}\nu\, \partial_t\Big(\phi_\varepsilon\big(|\Rot \bs h_\varepsilon|^p-\Psi^p\big)\Big)\\
=
\tfrac1p\int_\Omega\nu \phi_\varepsilon\big(|\Rot \bs h_\varepsilon(t)|^p-\Psi(t)^p\big)-
\tfrac1p\int_\Omega\nu \phi_\varepsilon\big(|\Rot \bs h_0|^p-\Psi(0)^p\big)\\
\geq \tfrac{a_*}{p}\norm{\Rot \bs h_\varepsilon(t)}^p_{\bs L^p(\Omega)}-\tfrac{a^*}{p}\norm{\Psi(t)}^p_{L^p(\Omega)}.
\end{multline}

The H\"older inequality allows us to obtain

\begin{multline}\label{PAEstdth3}
\int_{Q_t}\nu k_\varepsilon (|\Rot\bs h_\varepsilon|^p-\Psi^p)\Psi^{p-1}\partial_t\Psi\\
\leq a^*\norm{k_\varepsilon (|\Rot\bs h_\varepsilon|^p-\Psi^p)}_{L^1(Q_T}
\norm{\Psi^{p-1}}_{L^\infty(Q_T)}
\norm{\partial_t\Psi}_{L^\infty(Q_T)}
\end{multline}
and
\begin{equation}\label{PAEstdth4}
\int_{Q_t}\bs f\cdot\partial_t\bs h_\varepsilon
\leq \norm{\bs f}_{\bs L^2(Q_T)} \norm{\partial_t\bs h_\varepsilon}_{\bs L^2(Q_T)}.
\end{equation}

Arguing as in (\ref{gdth}) and (\ref{gdth2}) we have
\begin{multline}\label{PAEstdth6}
\int_{\Sigma_t} \bs g\cdot\partial_t\bs h_\varepsilon
=
\int_\Gamma \bs g(t)\cdot\bs h_\varepsilon(t)-
\int_\Gamma \bs g(0)\cdot\bs h_0-
\int_{\Sigma_t} \partial_t\bs g\cdot\bs h_\varepsilon\\
\leq
\tfrac{C_1}{p'}\norm{\bs g}^{p'}_{L^\infty(0,T;\bs L^{r'}(\Gamma))}+\frac1p\norm{\Rot\bs h_\varepsilon(t)}^p_{\bs L^p(\Omega)}
+
\norm{\bs g(0)}_{\bs L^{r'}(\Gamma)}\norm{\bs h_0}_{\bs L^p(\Gamma)}\\
+C_1\norm{\partial_t\bs g}_{\bs L^{r'}(\Sigma_T)}\norm{\Rot\bs h_\varepsilon}_{\bs L^p(Q_T)}.
\end{multline}

Using the Proposition \ref{te_proposition}, the relations (\ref{PAEstdth2}-\ref{PAEstdth6}) in the equality (\ref{PAEstdth1}) we obtain the
lemma.
\end{proof}

\begin{theorem}
The variational inequality {\em(\ref{iv})} has a unique solution $\bs h$ belonging to
$L^p(0,T;\W{\infty}{\Omega})\cap H^1(0,T;\bs L^2(\Omega))$.
\end{theorem}
\begin{proof} By the Lemmas \ref{PApEstRoth}, \ref{PApEstk} and \ref{PAEstdth} and well-known compactness results (see \cite{Simon1987}), there exists a subsequence $\varepsilon\rightarrow0$ such that
\begin{alignat*}{2}
\bs h_\varepsilon&\lraup\bs h&\quad&\mbox{in }L^\infty(0,T;\bs L^2(\Omega))\text{-weak}*\ \text{and}\ C([0,T];\bs L^1(\Omega))\text{-strong},\\
\Rot\bs h_\varepsilon&\lraup\Rot\bs h&\quad&\mbox{in }\bs L^p(Q_T)\mbox{-weak},\\
\partial_t\bs h_\varepsilon&\lraup\partial_t\bs h&\quad&\mbox{in }\bs L^2(Q_T)\mbox{-weak}.
\end{alignat*}

By the monotonicity of $k_\varepsilon$, choosing $\bs\varphi\in\K(t)$ we obtain
\begin{multline*}
\int_\Omega \nu k_\varepsilon\big(|\Rot \bs h_\varepsilon(t)|^p-\Psi(t)^p\big)|\Rot \bs h_\varepsilon(t)|^{p-2}\Rot \bs h_\varepsilon(t)\cdot\Rot(\bs\varphi-
\bs h_\varepsilon(t))\\
\leq \int_\Omega \nu |\Rot \bs \varphi|^{p-2}\Rot \bs \varphi\cdot\Rot(\bs\varphi-
\bs h_\varepsilon(t)).
\end{multline*}

Choosing in (\ref{ProblemaAproximado}) for test function $\bs\varphi-\bs h_\varepsilon(t)$, being $\bs\varphi\in\K(t)$ and
integrating in time, we obtain
\begin{multline*}
\int_{Q_T}\partial_t\bs h_\varepsilon \cdot(\bs\varphi-\bs h_\varepsilon) +
\int_{Q_T} \nu |\Rot \bs\varphi|^{p-2}\Rot \bs\varphi\cdot\Rot(\bs\varphi-\bs h_\varepsilon)\\
\geq\int_{Q_T}\bs f\cdot(\bs\varphi-\bs h_\varepsilon)
+ \int_{\Sigma_T}\bs g\cdot(\bs\varphi-\bs h_\varepsilon).
\end{multline*}

Noting that
\begin{align*}
\liminf_{\varepsilon\rightarrow0}\int_{Q_T}\partial_t\bs h_\varepsilon \cdot(\bs\varphi-\bs h_\varepsilon)
&=\liminf_{\varepsilon\rightarrow0}\bigg(\int_{Q_T}\partial_t\bs h_\varepsilon \cdot\bs\varphi -
\frac12\int_\Omega |\bs h_\varepsilon(t)|^2+\frac12\int_\Omega |\bs h_0|^2\bigg)\\
&\leq \int_{Q_T}\partial_t \bs h\cdot(\bs\varphi-\bs h),
\end{align*}
we get
\begin{multline*}
\int_{Q_T}\partial_t\bs h \cdot(\bs\varphi-\bs h) +
\int_{Q_T} \nu |\Rot \bs\varphi|^{p-2}\Rot \bs\varphi\cdot\Rot(\bs\varphi-\bs h_\varepsilon)\\
\geq\int_{Q_T}\bs f\cdot(\bs\varphi-\bs h)
+ \int_{\Sigma_T}\bs g\cdot(\bs\varphi-\bs h).
\end{multline*}

Assuming that $\bs h(t)\in\K(t)$ for a.e. $t\in\,(0,T)$ (this fact will be proved in the next lemma),
applying a variant of Minty's Lemma and standard arguments, we conclude
\begin{multline*}
\int_\Omega\partial_t\bs h(t)\cdot(\bs\varphi-\bs h)+ \int_\Omega \nu |\Rot \bs h|^{p-2}\Rot \bs h(t)\cdot\Rot(\bs\varphi-
\bs h(t))\\
\ge \int_\Omega\bs f(t)\cdot(\bs\varphi-\bs h(t))
+ \int_\Gamma\bs g(t)\cdot(\bs\varphi-\bs h(t)),
\end{multline*}
where $\bs\varphi$ is any function belonging to $\K(t)$, for a.e. $t\in\,(0,T)$.

The uniqueness is immediate.
\end{proof}

\begin{lemma} Let $\bs h_\varepsilon$ be the solution of the problem {\em (\ref{ProblemaAproximado})} and $\bs h$ the weak limit of a subsequence of $\{\bs h_\varepsilon\}_\varepsilon$ in $L^p(0,T;\Wtilde{p}{\Omega})\cap H^1(0,T;\bs L^2(\Omega))$. Then
 $$\bs h(t)\in\K(t), \text{ for a.e.\ }t\in(0,T).$$
\end{lemma}
\begin{proof} Define
\begin{align*}
A_\varepsilon&=\big\{(x,t)\in Q_T: |\Rot\bs h_\varepsilon(x,t)|^p-\Psi^p(x,t)|<\sqrt{\varepsilon}\big\},\\
B_\varepsilon&=\big\{(x,t)\in Q_T:\sqrt{\varepsilon}\le |\Rot\bs h_\varepsilon(x,t)|^p-\Psi^p(x,t)\le\tfrac1\varepsilon\big\},\\
C_\varepsilon&=\big\{(x,t)\in Q_T:|\Rot\bs h_\varepsilon(x,t)|^p-\Psi^p(x,t)>\tfrac1\varepsilon\big\}.
\end{align*}

We have
\begin{equation*}
\Tende{\displaystyle\int_{A_\varepsilon}\sqrt\varepsilon\leq \sqrt\varepsilon|Q_T|}{\varepsilon\rightarrow 0}{0}
\end{equation*}
and
\begin{equation*}
\Tende{\displaystyle\int_{C_\varepsilon}\frac1\varepsilon\leq \frac1\varepsilon
\int_{C_\varepsilon}\frac{k_\varepsilon(|\Rot\bs h_\varepsilon(x,t)|^p-\Psi^p(x,t))}{e^{{1}/{\varepsilon^2}}}
\leq C\,\frac1{\varepsilon}\, e^{{-1}/{\varepsilon^2}}}{\varepsilon\rightarrow 0}{0}
\end{equation*}

Recalling that, in $B_\varepsilon$, $k_\varepsilon(|\Rot\bs h_\varepsilon(x,t)|^p-\Psi^p(x,t))\ge e^{1/{\sqrt{\varepsilon}}}$, we have
$$|B_\varepsilon|=\int_{A_\varepsilon}1\le\int_{A_\varepsilon}\frac{k_\varepsilon(|\Rot\bs h_\varepsilon(x,t)|^p-\Psi^p(x,t))}{e^{1/{\sqrt{\varepsilon}}}}
\le Ce^{-1/{\sqrt{\varepsilon}}},$$
and so,  $\Tende{\displaystyle{|B_\varepsilon|}}{\varepsilon\rightarrow 0}{0}$.

Since
\begin{multline*}
\int_{Q_T}\big(|\Rot\bs h|-\Psi\big)^+=\int_{Q_T}\liminf_\varepsilon \big(|\Rot\bs h_\varepsilon|-\Psi\big)\wedge\tfrac1\varepsilon\vee\sqrt\varepsilon\\
\le
\liminf_\varepsilon \int_{Q_T}\big(|\Rot\bs h_\varepsilon|-\Psi\big)\wedge\tfrac1\varepsilon\vee\sqrt\varepsilon\\
=\liminf_\varepsilon\left(\int_{A_\varepsilon}\sqrt\varepsilon+
 \int_{B_\varepsilon}\big(|\Rot\bs h_\varepsilon|-\Psi\big)+
\int_{C_\varepsilon}\frac1\varepsilon\right) \\
=\liminf_\varepsilon \int_{Q_T}\big(|\Rot\bs h_\varepsilon|-\Psi\big)\chi_{_{B_\varepsilon}}\le\liminf_\varepsilon \|\,|\Rot\bs h_\varepsilon|-\Psi\|_{\bs L^p(Q_T)}\,
|B_\varepsilon|^{\frac1{p'}}=0,
\end{multline*}
we conclude that
$|\Rot\bs h|\le\Psi$ a.e. in $Q_T$, completing the proof.
\end{proof} 

\subsection{Continuous dependence on the data}

Consider given data $(\bs f_i,\bs g_i,\bs h_{i_0},\Psi_i)$ and define
$\K_i(t)=\{\bs v\in\W{p}{\Omega}:|\Rot\bs v|\le\Psi_i(t)\mbox{ a.e. in }\Omega\},$ for 
$i=1,2$.
\begin{lemma}\label{projeccao}
Given a function $\bs h_1\in L^p(0,T;\W{p}{\Omega})$ such that $\bs h_1(t)\in\K_1(t)$ for a.e. $t\in(0,T)$, there exists a function
$\widehat{\bs h}_2\in L^p(0,T;\W{p}{\Omega})$, verifying $\widehat{\bs h}_2(t)\in\K_2(t)$ for a.e. $t\in(0,T)$ and a positive constant $C$ such that
$$\|\Rot(\bs h_1-\widehat{\bs h}_2)\|_{\bs L^p(Q_T)}\le C\|\Psi_1-\Psi_2\|_{L^\infty(Q_T)}.$$
\end{lemma}

\begin{proof}
Define
$$\beta(t)=\|\Psi_1(t)-\Psi_2(t)\|_{L^\infty(\Omega)},\ \eta(t)= \frac\alpha{\alpha+\beta(t)}\ \mbox{ and }\ \widehat{\bs h}_2=\eta\bs h_1.$$

Then
\begin{equation*}
 |\Rot\widehat{\bs h}_2(t)|=\eta(t)|\Rot\bs h_1(t)|\le\eta(t)\Psi_1(t)\le\Psi_2(t),
\end{equation*}
since
\begin{equation*}
 \frac{\Psi_1(t)}{\Psi_2(t)}=\frac{\Psi_1(t)-\Psi_2(t)+\Psi_2(t)}{\Psi_2(t)}\le \frac{\beta(t)}{\alpha}+1=\frac1{\eta(t)},
\end{equation*}
and so $\widehat{\bs h}_2(t)\in\K_2(t)$ for a.e. $t\in(0,T)$.

Now,
\begin{gather}
\begin{split}
\|\Rot(\bs h_1(t)-\widehat{\bs h}_2(t))\|^p_{\bs L^{p}(\Omega)}&=\int_{\Omega}|\Rot(\bs h_1(t)-\widehat{\bs h}_2(t))|^p\\
&=\int_{\Omega}\big(1-\eta(t)\big)^p|\Rot\bs h_1(t)|^p\\
&=\int_{\Omega}\left(\frac{\beta(t)}{\alpha+\beta(t)}\right)^p|\Rot\bs h_1(t)|^p\\
&\le\frac{\beta(t)^p}{\alpha^p}\int_{\Omega}|\Rot\bs h_1(t)|^p.
\end{split}
\end{gather}

Integrating in time we obtain
\begin{equation*}
 \|\Rot(\bs h_1-\widehat{\bs h}_2)\|^p_{\bs L^p(Q_T)}
\le \frac{1}{\alpha^p}\|\Rot\bs h_1\|^p_{\bs L^p(Q_T)} \|\Psi_1-\Psi_2\|^p_{L^\infty(Q_T)}.
\end{equation*}
\end{proof}

\begin{remark} If we replace, in the last lemma, the subscript $1$ by the subscript $2$, the corresponding function we construct will be denoted by $\widehat{\bs h}_1$.
\hfill{$\square$}\end{remark}

\begin{theorem}
Let $\bs h_i$ denote the solution of the variational inequality {\em (\ref{iv})} with data $(\bs f_i,\bs g_i,\Psi_i,\bs h_{i_0})$, $i=1,2$.

Then there exists a positive constant $C$ such that
\begin{multline}
\|\bs h_1-\bs h_2\|^2_{L^\infty(0,T;\bs L^2(\Omega))}+\|\Rot(\bs h_1-\bs h_2)\|^{p\vee 2}_{\bs L^{p}(Q_T)}\le \\
C\big(\|\bs f_1-\bs f_2\|^{p'\wedge 2}_{\bs L^{q'}(Q_T)}+\|\bs g_1-\bs g_2\|^{p'\wedge 2}_{\bs L^{r'}(\Sigma_T)}\\
+\|\bs h_{1_0}-\bs h_{2_0}\|^{2}_{\bs L^{2}(\Omega)}+
\|\Psi_1-\Psi_2\|_{L^\infty(Q_T)}\big).
\end{multline}
\end{theorem}
\begin{proof} We know that, given $\bs\varphi\in\K_i(t)$, $i=1,2$, we have
\begin{align*}
\int_\Omega\partial_t \bs h_i(t)\cdot(\bs\varphi-\bs h_i(t))+ &
\int_\Omega\nu |\Rot\bs h_i(t)|^{p-2}\Rot\bs h_i(t)\cdot\Rot(\bs \varphi-\bs h_i(t))\\
\ge&\int_\Omega\, \bs f_i(t)\cdot(\bs\varphi-\bs h_i(t))+
\int_\Gamma\, \bs g_i(t)\cdot(\bs\varphi-\bs h_i(t)),\\
\bs h_i(0)=&\bs h_{i_0}.
\end{align*}

Choose, for $i=1$,  $\widehat{\bs h}_1$ as test function. Then,
\begin{multline*}
\int_\Omega\partial_t \bs h_1(t)\cdot(\bs h_1(t)-\widehat{\bs h}_1(t))+
\int_\Omega\nu |\Rot\bs h_1(t)|^{p-2}\Rot\bs h_1(t)\cdot\Rot(\bs h_1(t)-\widehat{\bs h}_1(t))\\
\le \int_\Omega\, \bs f_1(t)\cdot(\bs h_1(t)-\widehat{\bs h}_1(t))+
\int_\Gamma\, \bs g_1(t)\cdot(\bs h_1(t)-\widehat{\bs h}_1(t)),
\end{multline*}
from which we obtain
\begin{multline*}
\int_\Omega\partial_t \bs h_1(t)\cdot(\bs h_1(t)-\bs h_{2}(t))+
\int_\Omega\nu |\Rot\bs h_1(t)|^{p-2}\Rot\bs h_1(t)\cdot\Rot(\bs h_1(t)-\bs h_{2}(t))\\
\le \int_\Omega\, \bs f_1(t)\cdot(\bs h_1(t)-\bs h_{2}(t))+
\int_\Gamma\, \bs g_1(t)\cdot(\bs h_1(t)-\bs h_{2}(t))\\
+\int_\Omega\partial_t \bs h_1(t)\cdot(\widehat{\bs h}_1(t)-\bs h_2(t))+\int_\Omega\nu |\Rot\bs h_1(t)|^{p-2}\Rot\bs h_1(t)\cdot\Rot(\widehat{\bs h}_1(t)-\bs h_2(t))\\
+\int_\Omega\, \bs f_1(t)\cdot(\bs h_2(t)-\widehat{\bs h}_1(t))+
\int_\Gamma\, \bs g_1(t)\cdot(\bs h_2(t)-\widehat{\bs h}_1(t)).
\end{multline*}

We have an analogous expression with $\bs h_1$ substituted by $\bs h_2$ and $\widehat{\bs h}_1$ by $\widehat{\bs h}_2$. From both expressions we get
\begin{multline}\label{h1h2}
\int_\Omega\partial_t(\bs h_1(t)-\bs h_{2}(t))\cdot(\bs h_1(t)-\bs h_{2}(t))\\
+
\int_\Omega\nu\big(|\Rot\bs h_1(t)|^{p-2}\Rot\bs h_1(t) - |\Rot\bs h_2(t)|^{p-2}\Rot\bs h_2(t)\big)\cdot\Rot(\bs h_1(t)-\bs h_{2}(t))\\
\le \int_\Omega\, (\bs f_1(t)-\bs f_2(t))\cdot(\bs h_1(t)-\bs h_{2}(t))
+
\int_\Gamma\, (\bs g_1(t)-\bs g_2(t)\cdot(\bs h_1(t)-\bs h_{2}(t))+\Theta(t),
\end{multline}
where
\begin{multline*}
\Theta(t)=\int_\Omega\partial_t \bs h_1(t)\cdot(\widehat{\bs h}_1(t)-\bs h_2(t))+\int_\Omega\nu |\Rot\bs h_1(t)|^{p-2}\Rot\bs h_1(t)\cdot\Rot(\widehat{\bs h}_1(t)-\bs h_2(t))\\
+\int_\Omega\, \bs f_1(t)\cdot(\bs h_2(t)-\widehat{\bs h}_1(t))+
\int_\Gamma\, \bs g_1(t)\cdot(\bs h_2(t)-\widehat{\bs h}_1(t))\\
+\int_\Omega\partial_t \bs h_2(t)\cdot(\widehat{\bs h}_2(t)-\bs h_1(t))+\int_\Omega\nu |\Rot\bs h_2(t)|^{p-2}\Rot\bs h_2(t)\cdot\Rot(\widehat{\bs h}_2(t)-\bs h_1(t))\\
+\int_\Omega\, \bs f_2(t)\cdot(\bs h_1(t)-\widehat{\bs h}_2(t))+
\int_\Gamma\, \bs g_2(t)\cdot(\bs h_1(t)-\widehat{\bs h}_2(t)).
\end{multline*}

\begin{itemize}
\item  $p\ge 2$
\end{itemize}

From (\ref{h1h2}) we deduce,
 using the Remark \ref{PoincareTrace}, that there exists a positive constant $C_1$ such that
\begin{multline*}
\tfrac12\int_\Omega|\bs h_1(t)-\bs h_{2}(t)|^2+a_*\int_{Q_T}|\Rot(\bs h_1-\bs h_{2})|^p\\
\le C_1\Big(\|\bs f_1-\bs f_2\|_{\bs L^{q'}(Q_T)}
+\|\bs g_1-\bs g_2\|_{\bs L^{r'}(\Sigma_T)}\Big)\|\Rot(\bs h_1-\bs h_2)\|_{\bs L^p(Q_T)}\\
+\tfrac12\int_\Omega|\bs h_{1_0}-\bs h_{2_0}|^2+\int_0^T\Theta(\tau)d\tau.
\end{multline*}

It is easy to understand, by the expression of $\Theta$, that
\begin{align*}
\int_0^T \Theta(\tau)d\tau
&\leq C_2\big(\|\Rot(\bs h_1-\widehat{\bs h}_2)\|_{\bs L^p(Q_T)}+
\|\Rot(\bs h_2-\widehat{\bs h}_1)\|_{\bs L^p(Q_T)}\big)\\
&\leq C_3\norm{ \Psi_1- \Psi_2}_{\bs L^\infty(Q_T)},
\end{align*}
$C_2,\,C_3$ positive constants and, from this last inequality we conclude that
\begin{multline*}
\|\bs h_1-\bs h_{2}\|^2_{L^\infty(0,T;\bs L^2(\Omega))}+\|\Rot(\bs h_1-\bs h_{2})\|^p_{\bs L^p(Q_T)}\le C\left(\|\bs f_1-\bs f_2\|^{p'}_{\bs L^{q'}(Q_T)}\right.\\
\left.+\|\bs g_1-\bs g_2\|^{p'}_{\bs L^{r'}(\Sigma_T)}+\|\bs h_{1_0}-\bs h_{2_0}\|^2_{\bs L^2(\Omega)}+\|\Psi_1-\Psi_2\|_{L^\infty(Q_T)}\right).
\end{multline*}

\begin{itemize}
\item  $1<p<2$
\end{itemize}

Again, using (\ref{h1h2}),  the Remark \ref{PoincareTrace} and arguments similar to (\ref{rotp}), defining
$$
\widehat Q_T=\big\{(x,t)\in Q_T: \Rot\bs h_1(x,t) \neq \bs0,\,\Rot\bs h_2(x,t) \neq \bs0\big\}
$$
we obtain
\begin{multline*}
\tfrac12\int_\Omega|\bs h_1(t)-\bs h_{2}(t)|^2+a_*\Big(\int_{\widehat{Q}_T}|\Rot(\bs h_1-\bs h_2|^p\Big)^{\frac2{p}}\,\Big(\int_{\widehat{Q}_T}\big(|\Rot\bs h_1|+|\Rot\bs h_2|\big)^p\Big)^{\frac{p-2}p}\\
\le\|\bs f_1-\bs f_2\|_{\bs L^{q'}(Q_T)}\|\bs h_1-\bs h_2\|_{\bs L^q(Q_T)}
+ \|\bs g_1-\bs g_2\|_{\bs L^{r'}(\Sigma_T)}\|\bs h_1-\bs h_2\|_{\bs L^r(\Sigma_T)}\\
+\tfrac12\|\bs h_{1_0}-\bs h_{2_0}\|^2_{\bs L^2(\Omega)}+\int_0^T\Theta(\tau)d\tau.
\end{multline*}

So, there exists constants $C_2$ and $C_3$ such that
\begin{multline*}
\|\bs h_1-\bs h_{2}\|^2_{L^\infty(0,T;\bs L^2(\Omega))}+C_2\|\Rot(\bs h_1-\bs h_{2})\|^2_{\bs L^p(Q_T)}\\
\le C_3\left(\|\bs f_1-\bs f_2\|_{\bs L^{q'}(Q_T)}^2
+\|\bs g_1-\bs g_2\|^2_{\bs L^{r'}(\Sigma_T)}\right)
+\tfrac12\|\bs h_{1_0}-\bs h_{2_0}\|^2_{\bs L^2(\Omega)}+\int_0^T\Theta(\tau)d\tau
\end{multline*}
and the  conclusion follows as in the previous case.
\end{proof}

\subsection{The asymptotic behaviour in time of the solutions of the variational inequality}

Consider the  stationary variational inequality:
to find $\bs h_\infty\in\,\K_\infty$ such that
\begin{multline}\label{iv_infty}
\int_\Omega \nu |\Rot\bs h_\infty|^{p-2}\Rot\bs h_\infty \cdot\Rot(\bs\varphi -\bs h_\infty)\\
\geq\int_\Omega\bs f_\infty\cdot(\bs\varphi -\bs h_\infty)+\int_\Gamma \bs g_\infty\cdot(\bs\varphi -\bs h_\infty)\qquad
\forall\bs\varphi\in\K_\infty,
\end{multline}
where $\K_\infty=\big\{\bs v\in\W{p}{\Omega}: |\Rot\bs v|\leq\Psi_\infty\text{ a.e. in } \Omega\big\}$,
and assume
\begin{equation*}
 \bs f_\infty\in\bs L^{q'}(\Omega),\quad
\bs g_\infty\in\bs L^{r'}(\Gamma)
\quad\text{and}\quad
\Psi_\infty\in L^\infty(\Omega),\ \Psi_\infty\geq\alpha>0.
\end{equation*}

\begin{theorem}
Let $p\geq\frac65$, $\bs h$ be the solution of the variational inequality~{\em(\ref{iv})} and $\bs h_\infty$  the solution of the
~{\em(\ref{iv_infty})}.

Suppose that
\begin{align*}
&\bs f\in L^\infty(0,\infty;\bs L^{q'\vee 2}(\Omega)),\\
&\bs g\in L^\infty(0,\infty;\bs L^{r'}(\Gamma))\cap W^{1,r'}(0,\infty;\bs L^{r'}(\Gamma)),\\
&\Psi\in W^{1,\infty}(0,\infty; L^\infty(\Omega)).
\end{align*}

Suppose in addition that, for $\xi$ defined in {\em (\ref{xi})},
\begin{equation*}
\Tende{\displaystyle{\int_{\frac{t}2}^{t}}\xi(\tau)\,d\tau}
 {t\rightarrow\infty} {0\quad\text{if}\quad p>2\quad\text{and}}\quad\Tende{\displaystyle{\int_{t}^{t+1}}\xi(\tau)\,d\tau}
 {t\rightarrow\infty} {0\quad\text{if}\quad \frac65\le p\le 2}
\end{equation*}
and
\begin{equation*}
\exists D>0\quad\exists\gamma\quad\norm{\Psi(t)-\Psi_\infty}_{L^\infty(\Omega)}\leq\tfrac{D}{t^\gamma}\quad\text{with}\quad \gamma>
\begin{cases}
\frac32&\ \text{if}\ p>2,\\
\frac12&\ \text{if}\ \frac65\leq p\leq 2.
\end{cases}
\end{equation*}

Then we have
 \begin{equation*}
  \Tende{\norm{\bs h(t)-\bs h_\infty}_{\bs L^2(\Omega)}}{t\rightarrow\infty}{0}.
 \end{equation*}
\end{theorem}

\begin{proof}Let
\begin{equation*}
\beta(t)=\|\Psi(t)-\Psi_\infty\|_{L^\infty(\Omega)},\quad\text{and}\quad\eta(t)=\frac\alpha{\alpha+\beta(t)}.
\end{equation*}

Define
\begin{equation}\label{overline-h}
\overline{\bs h}(t)=\eta(t)\bs h_\infty,\quad\text{and}\quad\overline{\bs h}_\infty(t)=\eta(t)\bs h(t).
\end{equation}

As in Lemma~\ref{projeccao} we have $\overline{\bs h}(t)\in\K_\infty$ and $\overline{\bs h}_\infty(t)\in\K(t)$, for a.e. $t\in(0,\infty)$.

Substituting, in (\ref{h1h2}), $\bs h_1$ by $\bs h$ and $\bs h_2$ by $\bs h_\infty$, we obtain
\begin{multline*}
\int_\Omega\partial_t(\bs h(t)-\bs h_\infty)\cdot(\bs h(t)-\bs h_\infty)\\
+
a_* \int_\Omega\big(|\Rot\bs h(t)|^{p-2}\Rot\bs h(t)
-|\Rot\bs h_\infty)|^{p-2}\Rot\bs h_\infty\big)\cdot\Rot(\bs h(t)-\bs h_\infty)\\
\le \int_\Omega\, (\bs f(t)-\bs f_\infty)\cdot(\bs h(t)-\bs h_\infty)
+
\int_\Gamma\, (\bs g(t)-\bs g_\infty)\cdot(\bs h(t)-\bs h_\infty)+\Theta(t),
\end{multline*}
where
\begin{multline*}
\Theta(t)=\int_\Omega\partial_t \bs h(t)\cdot(\overline{\bs h}_\infty(t)-\bs h_\infty)+\int_\Omega\nu |\Rot\bs h(t)|^{p-2}\Rot\bs h(t)\cdot
\Rot(\overline{\bs h}_\infty(t)-\bs h_\infty)\\
+\int_\Omega\, \bs f(t)\cdot(\bs h_\infty-\overline{\bs h}_\infty(t))+
\int_\Gamma\, \bs g(t)\cdot(\bs h_\infty-\overline{\bs h}_\infty(t))\\
+\int_\Omega\nu |\Rot\bs h_\infty|^{p-2}\Rot\bs h_\infty
\cdot\Rot(\overline{\bs h}(t)-\bs h(t))\\
+\int_\Omega\, \bs f_\infty\cdot(\bs h(t)-\overline{\bs h}(t))+
\int_\Gamma\, \bs g_\infty\cdot(\bs h(t)-\overline{\bs h}(t))
\end{multline*}
and $\overline{\bs h}(t)$ and  $\overline{\bs h}_\infty(t)$ are defined in (\ref{overline-h}).

From Lemma~\ref{PAEstdth}, we observe that there exists positive constants, $C_1$ and $C_2$, independent of $t$, such that
\begin{equation*}
\|\partial_t\bs h\|_{\bs L^2(\Omega\times(0,t))}\leq C_1t^\frac12+C_2.
\end{equation*}

Define $\Phi(t)=\displaystyle{\int_\Omega\|\bs h(t)-\bs h_\infty\|^2}$.

Arguing as in Section~\ref{ca-equacao}, for $p>2$, we obtain, for a positive constant $C$,
\begin{equation*}
\Phi'(t)+C\Phi^\frac{p}2(t)\leq l(t),
\end{equation*}
where, for a positive constant $C_3$,
\begin{multline*}
l(t)=C_3\left(\|\bs f(t)-\bs f_\infty\|_{\bs L^{q'}(\Omega)}^{p'}+\|\bs g(t)-\bs g_\infty\|_{\bs L^{r'}(\Gamma)}^{p'}+
\|\Psi(t)-\Psi_\infty\|_{ L^{\infty}(\Omega)}\right)\\
+C_1t^\frac12\|\Psi(t)-\Psi_\infty\|_{ L^{\infty}(\Omega)}.
\end{multline*}

But
\begin{multline*}
\int_{\frac t2}^t l(\tau)\,d\tau\\
\leq
C_3 \int_{\frac t2}^t\left(\|\bs f(\tau)-\bs f_\infty\|_{\bs L^{q'}(\Omega)}^{p'}+\|\bs g(\tau)-\bs g_\infty\|_{\bs L^{r'}(\Gamma)}^{p'}+
\|\Psi(\tau)-\Psi_\infty\|_{ L^{\infty}(\Omega)}\right)\,d\tau\\
\Tende{+C_1\displaystyle\int_{\frac t2}^t \tau^\frac12\|\Psi(\tau)-\Psi_\infty\|_{ L^{\infty}(\Omega)}\,d\tau}{t\rightarrow\infty}{0,}
\end{multline*}
because $\gamma>\frac32$.

For $\frac65\leq p\leq2$ we have
\begin{equation*}
\Phi'(t)+C\Phi(t)\leq l(t),
\end{equation*}
where, for a positive constant $C_3$,
\begin{multline*}
l(t)=C_3\left(\|\bs f(t)-\bs f_\infty\|_{\bs L^{q'}(\Omega)}^{2}+\|\bs g(t)-\bs g_\infty\|_{\bs L^{r'}(\Gamma)}^{2}+
\|\Psi(t)-\Psi_\infty\|_{ L^{\infty}(\Omega)}\right)\\
+C_1t^\frac12\|\Psi(t)-\Psi_\infty\|_{ L^{\infty}(\Omega)}.
\end{multline*}

But
\begin{multline*}
\int_{t}^{t+1} l(\tau)\,d\tau\\
\leq
C_3 \int_{t}^{t+1}\left(\|\bs f(\tau)-\bs f_\infty\|_{\bs L^{q'}(\Omega)}^{2}+\|\bs g(\tau)-\bs g_\infty\|_{\bs L^{r'}(\Gamma)}^{2}+
\|\Psi(\tau)-\Psi_\infty\|_{ L^{\infty}(\Omega)}\right)\,d\tau\\
\Tende{+C_1\displaystyle\int_{t}^{t+1} \tau^\frac12\|\Psi(\tau)-\Psi_\infty\|_{ L^{\infty}(\Omega)}\,d\tau}{t\rightarrow\infty}{0,}
\end{multline*}
because $\gamma>\frac12$.

Arguing, in both cases, exactly as in the Section~\ref{ca-equacao}, the conclusion follows.
\end{proof}

\end{document}